\RequirePackage{fix-cm}
\documentclass[smallcondensed,leqno,numbook,envcountsame]{svjour3}     
\smartqed  
 \usepackage{mathptmx}      
\usepackage{amssymb}
\usepackage{mathtools}
\usepackage{enumerate}
\usepackage[abbrev]{amsrefs}

\newcommand{\N}{\mathbb{N}}

\newcommand{\R}{\mathbb{R}}

\newcommand{\cA}{\mathcal{E}}
\newcommand{\cE}{\mathcal{E}^0}

\newcommand{\gm}{\gamma}\newcommand{\dl}{\delta}
\newcommand{\eps}{\varepsilon}

\newcommand{\sg}{\sigma}

\newcommand{\bone}{\mathbf{1}}

\newcommand{\wg}{\wedge}

\newcommand{\qad}{\phantom{\le{}}}
\newcommand{\Capa}{{\text{\rm Cap}}}
\newcommand{\ob}[1]{\mkern 1.5mu\overline{\mkern-1.5mu#1\mkern-1.5mu}\mkern 1.5mu}

\newcommand{\relmiddle}[1]{\mathrel{}\middle#1\mathrel{}}
\DeclareMathOperator*{\essinf}{ess\,inf}
\DeclareMathOperator*{\esssup}{ess\,sup}
\newcommand{\myqedhere}{\tag*{\rlap{\hbox to\displaywidth{\hfill\qed}}}}
\allowdisplaybreaks[4]

\begin{document}
\spnewtheorem{condition}[theorem]{Condition}{\it}{\rm}
\spnewtheorem*{xproof}{}{\itshape}{\rmfamily}
\newcommand\xprooftitle{}
\renewenvironment{proof}[1][\proofname]
 {\renewcommand\xproofname{#1}\xproof}
 {\endxproof}

\title{Doubly Feller property of Brownian motions with Robin boundary condition
}

\titlerunning{Robin boundary condition}        

\author{
        Kouhei Matsuura 
}

\authorrunning{K. Matsuura} 

\institute{
          K. Matsuura\at
              Mathematical Institute, Tohoku University, Aoba, Sendai 980--8578, Japan\\
         \email{kouhei.matsuura.r3@dc.tohoku.ac.jp} 
}

\date{Received: date / Accepted: date}

\maketitle
\begin{abstract}
In this paper, we consider first order Sobolev spaces with Robin boundary condition on unbounded Lipschitz domains. Hunt processes  are associated with these spaces. We prove that the semigroup of these processes are doubly Feller. As a corollary, we provide a condition for semigroups generated by these processes being compact.
\end{abstract}

\section{Introduction}
Let $D$ be a connected open subset of $\R^d$. Let us denote by $H_{0}^{1}(D)$ and $H^{1}(D)$ the first order $L^2$ Sobolev space on $D$ with Dirichlet and Neumann boundary condition, respectively. For $p \ge 1$, we denote by $L^{p}(D)$ the set of functions on $D$ which are $p$-th integrable with respect to the $d$-dimensional Lebesgue measure $m$. Let us assume $D$ is {\it thin at infinity}, i.e. $\lim_{\,x \in D,\,|x| \to \infty}m(B(x,1) \cap D)=0$.
Here, $B(x,1)$ is the open ball of $\R^d$ centered at $x \in \R^d$ with radius $1$.
Then, we can show that the embedding $H_{0}^{1}(D) \subset L^{2}(D)$ is compact (\cite[Chapter~V. Remark~5.18~(4)]{EE}). On the other hand, even if $D$ is thin at infinty, the embedding $H^{1}(D) \subset L^{2}(D)$ is not necessarily compact. In fact, if the Lebesgue measure of $D$ is infinite, it is known that the embedding is not compact (\cite[Theorem~6.42]{Ad}). 

Let us consider the Sobolev space with the following boundary condition: $$\partial f/\partial \mathcal{N}+\beta f=0 \text{ on } \partial D,$$ where $\mathcal{N}$ is the outward unit normal vector on the boundary $\partial D$ and $\beta$ is a nonnegative measurable function on $\partial D$. This type of boundary condition is called Robin boundary condition. The associated Dirichlet space on $L^{2}(D)$ is expressed as
\begin{align*}
\cA(f,g)&=\frac{1}{2}\int_{D}(\nabla f, \nabla g)\,dx+\int_{\partial D}\tilde{f} \tilde{g}\,\beta d\sg,\quad f,g \in \mathcal{D}(\cA), \\
\mathcal{D}(\cA)&=\left\{ f \in H^{1}(D) \relmiddle| \int_{\partial D}\tilde{f}^2\,\beta d\sg<\infty \right\},
\end{align*}
where $\sg$ is the $(d-1)$-dimensional Hausdorff measure on $\R^d$ restricted to $\partial D$, the topological boundary of $D$. $\tilde{f}$ is a suitable version of $f \in H^{1}(D)$. See Theorem~\ref{thm:bound} below for the definition of $\tilde{f}$. The domain $\mathcal{D}(\cA)$ is a superset and subset of $H_{0}^{1}(D)$ and $H^{1}(D)$, respectively. Therefore, for a domain $D$ with $m(D)=\infty$, it is not clear whether $\mathcal{D}(\cA) \subset L^{2}(D)$ is compactly embedded in $L^2(D)$ or not.
 
 Arendt and Warma \cite{AW} invesitigate the compact embedding $\mathcal{D}(\cA) \subset L^{2}(D)$ for open sets with finite volume.
In this paper, we check a probabilistic condition for the compact embedding $\mathcal{D}(\cA) \subset L^{2}(D)$ when $D$ is thin at infinity and $\partial D$ is sufficiently smooth and $\beta$ is non-degenerate and locally bounded. Non-degenerate means $\sigma\text{-}\essinf_{z \in \partial D}\beta(z)>0$. Here, $\sg\text{-}\essinf$ denotes the essential infimum with respect to $\sg$.
Let $E$ be a locally compact separable metric space and $\mu$ a Radon measure with topological full support. Takeda \cite{T} prove that if the Hunt process on $E$ generated by a regular Dirichlet form is {\it irreducible}, {\it resolvent strong Feller}, in addition, has a {\it tightness property}, then the domain of Dirichlet form is compactly embedded in $L^2(E,\mu)$. This is equivalent to the semigroup of this Hunt process becomes a compact operator on $L^2(E,\mu)$.

We briefly explain how to check these properties. First, we construct a Hunt process $X^0=(\{X_{t}^0\}_{t \ge 0}, \{P_x\}_{x \in \ob{D}})$ on $\ob{D}$ with a kind of resolvent strong Feller property: 
\begin{align}
R_{\alpha}^0(L^{1}(\ob{D},m) \cap L^{\infty}(\ob{D},m))\subset C_{b}(\ob{D}) \label{eq:wrsf}
\end{align}
 (Theorem~\ref{thm:1}, Theorem~\ref{thm:wrsf}). Here, $\{ R^0_{\alpha}\}_{\alpha>0}$ denotes the resolvent of $X^0$. This is Hunt process corresponding to $H^{1}(D)$ which is a regular Dirichlet form on $L^{2}(\ob{D},m)$. We should call $X^0$ a reflecting Brownian motion on $\ob{D}$. For the proof of \eqref{eq:wrsf}, we employ a PDE methods due to Stampacchia and Moser, as used in \cite{FT0, FT}. Next, we prove the smoothness of $\sg$ (Proposition~\ref{prop:1}).  This allows us to define the positive continuous additive functional $\{L_t\}_{t \ge 0}$ with Revuz measure $\sg$.
Since $\{L_t\}_{t \ge 0}$ increases only when $X^0$ hits the boundary $\partial D$, $\{ L_t\}_{t \ge0}$ is said to be the {\it boundary local time} of $X^0$. We define the subprocess $Y$ of $X^0$ by the multiplicative functional $\{ \exp(-\int_{0}^{t}\beta(X_t^0)\,dL_t) \}_{t \ge 0}$. The Dirichlet form of $Y$ is identified with $(\cA,\mathcal{D}(\cA))$. Hence, for the proof of compactness, it suffices to prove $Y$ has the three properties stated above.
 
We explain how to check the resolvent strong Feller property of $Y$. Note that, for any $f \in L^{1}(\ob{D},m) \cap L^{\infty}(\ob{D},m)$, we can obtain the following inequality:
\begin{align*}
&|R_{\alpha}f(x)-\gamma R_{\gamma+\alpha}^{0}R_{\alpha}f(x)| \\
&\le \| f\|_{\infty} \times (\gamma+\alpha)^{-1}+\|f\|_{\infty}\times (1/\alpha) \times \int_{0}^{1}E_{x}\left[1-\exp \left(-\int_{0}^{-\frac{\log s}{\gamma}}\beta (X_u^0 ) \,dL_{u} \right) \right]\,ds.
\end{align*}
Here, $\{R_{\alpha}\}_{\alpha>0}$ denotes the resolvent of $Y$. Therefore, if for any compact subset $K$ of $\ob{D}$,  
\begin{equation}
\lim_{t \to 0}\sup_{x \in K}E_{x} \left[1-\exp \left(-\int_{0}^{t}\beta(X_t^0)\,dL_t \right) \right]=0 \label{eq:eqblt},
\end{equation}
we can see $R_{\alpha}(L^{1}(\ob{D},m) \cap L^{\infty}(\ob{D},m))\subset C_{b}(\ob{D})$ because $\{ R_{\alpha}^0\}_{\alpha>0}$ posseses the property \eqref{eq:wrsf}. For a bounded Lipschitz domain $D$, using an upper Gaussian estimate of the heat kernel of $X^0$, we can prove 
\begin{equation}
\sup_{x \in \ob{D}}E_{x}[L_t] =O(\sqrt{t}) \text{ as } t \to 0\label{eq:eqblt2}.
\end{equation}
If $\beta$ is locally bounded, it follows from Jensen's inequality that \eqref{eq:eqblt2} implies \eqref{eq:eqblt}. However, for an unbounded domain  thin at infinity, the heat kernel of $X^0$ does not satisfy the Gaussian estimate in general. For this reason, we need to make a different approach.

We take bounded increasing open subsets $\{U_n\}_{n=1}^{\infty}$  of $\R^d$ such that $\ob{D}=\bigcup_{n=1}^{\infty} \ob{D} \cap U_n$. Then, the left hand side of \eqref{eq:eqblt} is estimated as follows:
\begin{align*}
\sup_{x \in K}E_{x}\left[1-\exp \left(-\int_{0}^{t}\beta(X_s^0)\,dL_s \right) \right]
&\le \sup_{x \in K}E_{x}\left[\int_{0}^{t \wg \tau_n} \beta(X_s^0) \,dL_{s} \right]+\sup_{x \in K}P_{x}[ t \ge \tau_n],
\end{align*}
where $\tau_n$ is the first leaving time of $X^0$ from $\ob{D} \cap U_n$.
Let us denote by $X^n$ the part of $X^0$ on $\ob{D} \cap U_n$. Then, $\{ L_{t \wg \tau_n} \}_{t \ge 0}$ is regarded as a boundary local time of  $X^n$ and thus $\{ L_{t \wg \tau_n} \}_{t \ge 0}$ increases only when $X^n_{t} \in \partial D \cap U_n$. Therefore, if the boundary of $U_n \cap D$ is smooth enough and $\beta$ is locally bounded, as in the case when $D$ is bounded Lipschitz domain, the estimate $$\sup_{x \in K}E_{x}\left[\int_{0}^{t \wg \tau_n} \beta(X_s^0) \,dL_{s} \right] \le c_n \sqrt{t}$$ is expected for any $0<t \le 1$. Here $c_n$ denotes a positive constant depending on $D \cap U_n$ and $\beta$. In fact, this is true (Lemma~\ref{blt2}, \ref{blt3}). On the other hand, it is difficult to obtain a quantative estimate of $P_{x}[ t \ge \tau_n]$ because the upper bound of the heat kernel of $X^0$ is unknown. To overcome this difficulty we prove the semigroup strong Feller property of part processes $\{X^n\}$, which yields $\lim_{n \to \infty}\sup_{x \in K}P_{x}[ t \ge \tau_n]=0$ (Lemma~\ref{exit1}) and completes the proof of \eqref{eq:eqblt} (Proposition~\ref{blt1}). To show the semigroup strong Feller property of $\{X^n\}$, we employ the theory of extension domains (Lemma~\ref{lem:identification}). By applying this result, we can strengthen the statement in \eqref{eq:wrsf} as follows: The semigroups of $X^0$ and $Y$ have the semigroup strong Feller property (Theorem~\ref{thm:doubf}~(i), Theorem~\ref{cor:corstf}). Chung's method for proving the strong Feller property of part processes is well known (see \cite[Theorem~1, 2, 3, and Corollary]{C} for details). However, we do not use his result because $X^0$ does not generally satisfy $R_{\alpha}^0(C_{\infty}(\ob{D})) \subset C_{\infty}(\ob{D})$, required in \cite{C}. Here, $C_{\infty}(\ob{D})$ denotes the continuous functions on $\ob{D}$ vanishing at infinity.
 
We explain how to prove tightness property of $Y$.
First, we derive the heat kernel estimate of $Y$ by using  Maz'ya's inequality (Lemma~\ref{lem:mazja}, Theorem~\ref{thm:hke}) when $\sg\text{-}\essinf_{z \in \partial D} \beta(z)>0$. Using this estimate, we can prove the Feller property of $Y$ (Theorem~\ref{thm:doubf}~(ii)). Then, tightness property is equivalent to $\lim_{|x| \to \infty} R_{\alpha}\bone_{\ob{D}}(x)=0$. Because the semigroups of $Y$ is ultracontractive, it is not too difficult to prove $\lim_{|x| \to \infty} R_{\alpha}\bone_{\ob{D}}(x)=0$ under the thin at infinity condition. The irreducible property of $Y$ is clear in our setting.
 
 This paper organized as follows. In Section~2, we introduce a framework and state the main theorems. 
Section~3 provides the proof of a kind of resolvent strong Feller property of $X^0$ (Theorem~\ref{thm:1}, Theorem~\ref{thm:wrsf}). 
In Section~4, we introduce the concept of relative capacity and results on it. We also discuss the smoothness of the boundary measure $\sg$ (Proposition~\ref{prop:1}). Section~3 provides the proof of doubly semigroup strong Feller property of $Y$ (Theorem~\ref{thm:doubf}~(ii)). We also prove the compactness of the embedding  $\mathcal{D}(\cA) \subset L^{2}(D)$ (Corollary~\ref{thm:3}).
In the last section, we prove some auxiliary propositions  and discuss the conditions imposed on the theorems.
\\

\noindent
{\bf Notation.} Throughout this paper, we use the following notations.
\begin{itemize}
\item[(1)] Given a topological space $E$, the Borel $\sg$-algebra on $E$ is denoted by $\mathcal{B}(E)$. Let $\mu$ be a positive measure on the measurable space $(E,\mathcal{B}(E))$. For $p \in [1,\infty]$, the real $L^{p}$ space on the measure space $(E,\mathcal{B}(E),\mu)$ is denoted by $L^{p}(E,\mu)$, and its norm by $\|\cdot\|_{L^{p}(E,\mu)}$. The standard inner product on $L^{2}(E,\mu)$ is denoted by $(\cdot,\cdot)_{L^{2}(E,\mu)}$. For $f: E \to \R$, we write
$\|f\|_{\infty}=\sup_{x \in E}|f(x)|$. We also write 
\begin{align*}
&\mathcal{B}_{b}(E):=\{f:E \to \R \mid f \text{ is } \mathcal{B}(E)\text{ measurable, } \|f\|_{\infty}<\infty\},\\
&\mathcal{B}_{+}(E):=\{f:E \to \R \mid f \text{ is } \mathcal{B}(E)\text{ measurable and }f \ge 0 \text{ on }E\},\\
&C(E):=\{f:E \to \R \mid f \text{ is continuous}\},\,C_{b}(E):=C(E) \cap \mathcal{B}_{b}(E), \\ 
&C_{\infty}(E):=\{f\in C(E)\mid  \{x \in E \mid f(x) \ge \dl \} \text{ is compact for any } \dl \in (0,\infty)  \},\\
&C_{c}(E):=\{ f \in C(E) \mid \text{ support of }f\text{ is compact} \}.
\end{align*} 
\item[(2)]
 $d \ge 2$ is a positive integer. $B(x,R)$ denotes open ball of $\R^d$ centered at $x \in \R^d$ with radius $R>0$. If $x$ is the origin  of $\R^{d}$, we write $B(R)$ for $B(x,R)$ and $B_{+}(R)$  for $\{x=(x_{1},\ldots,x_d) \in \R^d \mid x \in B(R) \text{ and } x_d>0 \}$. For an open subset $E \subset \R^d$, $\ob{E}$ and $\partial E$ denote closure of $E$ in $\R^d$ and $\ob{E} \setminus E$, respectively. $\ob{E}$, $\partial E$ are regarded as topological subspaces of $\R^d$. The $d$-dimensional Lebesgue measure is denoted by $m$ or simply by $dx$. 
For an open subset $E \subset \R^d$, we denote by $W^{1,2}(E)$ the first order $L^2$-Sobolev space on $E$, that is,
\begin{align*}
W^{1,2}(E)&=\left\{ f \in L^{2}(E,m) \relmiddle|  \frac{\partial f}{\partial x_i}\in L^{2}(E,m), 1 \le i \le d\right\}, 
\end{align*}
where $\partial f/\partial x_i$ is the distributional derivative of $f$ on $E$. For $f\in W^{1,2}(E)$, we denote $\nabla f=(\partial f/\partial x_1,\ldots,\partial f/\partial x_d)$.
For  $f\in W^{1,2}(E)$, $ \| f\|_{W^{1,2}(E)} := \left( \int_{E}|\nabla f|^{2}\,dx+\int_{E}|f|^{2}\,dx \right)^{1/2}$ defines a norm on $W^{1,2}(E)$. $W^{1,2}(E)$ is sometimes written as $H^{1}(E)$. $H_{0}^{1}(E)$ and $\widetilde{H}^{1}(E)$ denote the closure of $C_{0}^{\infty}(E)$ in $H^{1}(E)$ and $H^{1}(E) \cap C(\ob{E})$ in $H^{1}(E)$, respectively. Here,  $C_{0}^{\infty}(E)$ denote the set of infinitely differentiable functions on $E$ with compact support. 
\end{itemize}

\section{Framework and main results}
Throughout this paper $D$ is a domain, that is, connected open subset of $\R^d$ $(d\ge 2)$. Then, we can define the following bilinear form on $L^{2}(D,m)$:
\begin{equation*}
\cE(f,g)=\frac{1}{2}\int_{D}(\nabla f,\nabla g)\,dx, \quad f,g \in H^{1}(D),
\end{equation*}
where 
$(\cdot,\cdot)$ is the standard inner product on $\R^d$.
$(\cE, H^{1}(D))$ becomes a strong local Dirichlet form on $L^{2}(D,m)$. 
Let $\{G_{\alpha}^{0}\}_{\alpha>0}$ be the $L^2$-resolvent associated with $(\cE,H^{1}(D))$. This section is devoted to the proof of the next theorem:
To construct a Hunt process associated with $(\cE, H^{1}(D))$ starting from every point of $\ob{D}$, we impose the following conditions on $D$: 
\begin{condition}\label{as:1}
\begin{enumerate}[({A}.1)]
\item[(A.1)] For any $a \in \partial D$, there are its neighbourhood $W_a$ in $\R^d$ and a bi-Lipschitz mapping  $\Psi_a: B(1) \to W_a$ such that $\Psi_{a}(0)=a$ and $\Psi_a(B_{+}(1))=W_a \cap D$. 
\item[(A.2)]  For any compact subset $K $ of $\partial D$, we have 
$$\sup_{a \in K}\max\{ \text{Lip}(\Psi_a), \text{Lip}(\Psi_{a}^{-1})\} \in (0,\infty),$$
where we define $\text{Lip}(\Psi_a)=\inf \{ L\ge 0 \mid |\Psi_{a}(x)-\Psi_{a}(y)| \le L|x-y|, x,y \in B(1) \}$. $\text{Lip} (\Psi_a^{-1})$ is defined in the same manner. 
\item[(A.3)] $(\cE,H^{1}(D))$ is a regular Dirichlet form on $L^{2}(\ob{D},m)$; that is, $H^{1}(D) \cap C_{c}(\ob{D})$ is both dense in $(H^{1}(D), \|\cdot \|_{H^{1}(D)})$ and in $(C_{c}(\ob{D}),\|\cdot\|_{\infty})$.
\end{enumerate}
\end{condition}
\begin{remark}
\begin{enumerate}
\item[(i)] It follows from  (A.1) that $m(\partial D)=0$.
\item[(ii)] If the boundary of $D$ is locally expressible as a graph of a continuous function of $d-1$ variables, we can check (A.3) (see \cite[Chapter~V, Theorem~4.7]{EE} for details). However, we don't know (A.1) and (A.2) imply (A.3).
\end{enumerate}
\end{remark}

\begin{theorem}\label{thm:wrsf}
Suppose Condition~(A.1) and (A.2). Let $\alpha>0$, $p>d$, $n \in \mathbb{N}$, and  $f \in L^{p}(\ob{D},m) \cap L^{2}(\ob{D},m)$. Then, $G_{\alpha}^{0}f $ is uniformly continuous on $D \cap B(n)$. In paticular, $G_{\alpha}^{0}f $ is continuous on $\ob{D}$.
\end{theorem}

Let $\Capa_{\ob{D}}$ be the capacity corresponding to the regular Dirichlet form $(\cE,H^{1}(D))$ on $L^{2}(\ob{D},m)$. We define the measure $\sg$ on $(\partial D, \mathcal{B}(\partial D))$  by $\sigma=\bone_{\partial D}\cdot  \mathcal{H}^{d-1}$, where $\mathcal{H}^{d-1}$ is the $(d-1)$-dimensional Hausdorff measure on $\R^d$. From Proposition~\ref{prop:1} below, we can check $\sg$ is a {\it smooth} Radon measure  on $(\partial D, \mathcal{B}(\partial D))$. Therefore, we can define a positive continuous additive functional $\{L_t\}_{t \ge 0}$ in the Revuz correspondence to  $\sg$.

\begin{proposition}\label{prop:1}
Suppose (A.1) and (A.3). Then, $\sg$ is a Radon measure on $(\partial{D}, \mathcal{B}(\partial{ D}) )$. Moreover, 
if $A \subset \partial{ D}$ satisfies $\Capa_{\ob{D}}(A)=0$, then $\sg(A)=0$.
\end{proposition}

In the sequel, we assume the following condition.

\begin{condition}\label{as:2}
For any compact subset $K$ of $\ob{D}$, there exists a bounded open subset $U$ of $\R^d$ such that $K \subset U$ and $U \cap D$ is a bounded Lipschitz domain of $\R^d$.
\end{condition}

In section~7, we provide the definition of bounded Lipschitz domain and prove that Condition~\ref{as:2} implies Condition~\ref{as:1}. Under Condition~2.5, we can construct a Hunt process associated with $(\cE,H^{1}(D))$ whose resolvent satisfies the same property as $\{G_{\alpha}^0\}_{\alpha>0}$: 

\begin{theorem}\label{thm:1}
Under Condition~\ref{as:2}, there exists a Hunt process $X^0=(\{X_{t}^0\}_{t \ge 0}, \{P_x\}_{x \in \ob{D}})$ associated with $(\cE, H^{1}(D))$ whose resolvent $\{R_{\alpha}^0\}_{\alpha>0}$ has the following property: for any $f \in L^{1}(\ob{D},m) \cap L^{\infty}(\ob{D},m)$ and $\alpha>0$, we have $R_{\alpha}^{0}f \in C_{b}(\ob{D})$.
\end{theorem}

Let $\beta$ is a nonnegative locally bounded Borel measurable function on $\partial D$ and $Y=(\{ Y_t\}_{t \ge 0}, \{ P_x^{\beta}\}_{x \in \ob{D}})$ be  the subprocess of $X^0$ defined by the multiplicative functional $$\left\{\int_{0}^{t}\exp(-\beta(X_s^0))\,dL_s\right\}_{t \ge 0}.$$ Namely, $Y$ is the Hunt process whose semigroup $\{p_t\}_{t>0}$ is given by
$$p_{t}f(x)=E_{x}\left[f(X_t^0)\int_{0}^{t}\exp(-\beta(X_s^0))\,dL_s\right],\quad t>0,\ x \in \ob{D},\ f \in \mathcal{B}_{b}(\ob{D}).$$
Under Condition~\ref{as:2} and the following condition on $\beta$, $Y$ has the doubly Feller property in Chung's sense \cite{C}.
 
\begin{theorem}\label{thm:doubf}
Suppose Condition~\ref{as:2}. Then, 
\begin{itemize}
\item[(i)]
for any $t>0$ and $f \in \mathcal{B}_{b}(\ob{D})$, $p_{t}f \in C_{b}(\ob{D})$.
\end{itemize}
Furtheremore, we assume $\sg\text{-}\essinf_{z \in \partial D} \beta(z)>0$. Then,
\begin{itemize}
\item[(ii)]
for any $t>0$ and $f \in C_{\infty}(\ob{D})$,  $p_{t}f \in C_{\infty}(\ob{D})$.
\end{itemize}
Here, $\sg\text{-}\essinf$ denotes the essential infimum with respect to $\sg$.
\end{theorem}

Let $(\cA, \mathcal{D}(\cA))$ be the Dirichlet form on $L^{2}(D,m)$ generated by $Y$.  Under Condition~\ref{as:2}, 
$(\cA, \mathcal{D}(\cA))$ is given by
\begin{align*}
\cA(f,g)&=\frac{1}{2}\int_{D}(\nabla f, \nabla g)\,dx+\int_{\partial D}\tilde{f} \tilde{g}\,\beta d\sg,\quad f,g \in \mathcal{D}(\cA), \\
\mathcal{D}(\cA)&=\left\{ f \in H^{1}(D) \relmiddle| \int_{\partial D}\tilde{f}^2\,\beta d\sg<\infty \right\},
\end{align*}
where $\tilde{f}$ is the $\Capa_{\ob{D}}$-quasi continuous version of $f \in H^{1}(D)$. See Definition~\ref{defn:relative}~(iii) for the definition.

\begin{corollary}\label{thm:3}
Suppose Conditions~\ref{as:2}, $\sg\text{-}\essinf_{z \in \partial D}\beta(z)>0$. If the domain $D$ satisfies 
$$\lim_{x \in \ob{D},\,|x| \to \infty}m(D \cap B(x,1)) \to 0,$$
the embedding $\mathcal{D}(\cA) \subset L^{2}(D,m)$ is compact.
\end{corollary}

\section{Proof of Theorem~\ref{thm:wrsf}}

Let $\{G_{\alpha}^{0}\}_{\alpha>0}$ be the $L^2$-resolvent associated with $(\cE,H^{1}(D))$. To prove Theorem~\ref{thm:wrsf}, we employ a PDE argument due to Stampacchia and Moser used in \cite{FT0,FT}. Throughout this section, we write $M(a)$ for $\max\{ \text{Lip}(\Psi_a), \text{Lip}(\Psi_{a}^{-1})\}$ $(a \in \partial D)$.
Futhermore, for $n \in \N$, we define
$$
D_n=D \cap B(n),\quad K_n=\partial D \cap \ob{B(n)},
 \quad M_{n}=\sup_{a \in K_n}M(a).
$$
From (A.2), we have $M_n \in(0, \infty)$.  
\subsection{Sobolev inequalities of Moser's type}
We note that a Sobolev inequality of Moser's type in \cite[Lemma~2]{Mo} is valid 
for $H^{1}(B_{+}(r))$:
\begin{lemma}\label{lem:moser}
For any $\kappa \in (0,1]$, there exists $c_{1}=c_{1}(d,\kappa)>0$ such that 
\begin{align*}
&\left(r^{-d}\int_{B_{+}(r)}|f|^{q}\,dx \right)^{1/q} 
\le c_{1}
\left\{\left( r^{2-d}\int_{B_{+}(r)}|\nabla f|^{2}\,dx \right)^{1/2}+\left( r^{-d}\int_{N_{1}}|f|^{2}\,dx\right)^{1/2} \right\},
\end{align*}
for all $r>0$, $f \in H^{1}(B_{+}(r))$ and $q \in [2,2d/(d-2)]$ ($q \in [2,\infty)$ if $d=2$). Here $N_{1}$ is a Lebesgue measurable subset of $B_{+}(r)$ with $m(N_{1}) \ge \kappa m(B_{+}(r))$.
\end{lemma}

In the following, for $a \in \partial D$, $0<r\le1$, we write $B_{a}^{\ast}(r)$ for $\Psi_{a}(B_{+}(r))$.
Since each $\Psi_{a}$ is bi-Lipschitz continuous function, the following estimate holds:
\begin{equation}
M(a)^{-d}m(B_{+}(r)) \le  m(B_{a}^{\ast}(r))\le M(a)^{d}m(B_{+}(r)),\quad  r \in (0,1] \label{eq:jacobian2}.
\end{equation}
\begin{lemma} \label{lem:moser2}
For any $\eta \in(0,1]$,  
$q \in [2,2d/(d-2)]$ ($q \in [2,\infty)$ if $d=2$), there exists $c_{2}=c_{2}(d,M_n,\eta, q)>0$ such that 
\begin{equation*}
\left(\int_{B_{a}^{\ast}(r)}|f|^{q}\,dx \right)^{1/q} \le c_{2}r^{d\left( \frac{1}{q}-\frac{1}{2}\right)} \left( r^{2} \int_{B_{a}^{\ast}(r)}|\nabla f|^{2}\,dx+\int_{N_{2}}|f|^{2}\,dx \right)^{1/2}
\end{equation*}
for all $f \in H^{1}(B_{a}^{\ast}(r))$, $N_{2} \subset B_{a}^{\ast}(r)$ with $m(N_{2}) \ge \eta m( B_{a}^{\ast}(r))$, $a \in K_n$, $r \in  (0,1]$.
\end{lemma}
\begin{proof}
Let $a \in K_n$, $r \in  (0,1]$, $\eta \in (0,1]$, and  $N_{2} \subset B_{a}^{\ast}(r)$ with  
$m(N_2) \ge \eta m(B_{a}^{\ast}(r))$. Set $N_{1}=\Psi_{a}^{-1}(N_2) \subset B_{+}(r)$. Using \eqref{eq:jacobian2}, we have
$$
m(N_1) 
\ge (M^{-2d}_{n} \eta \wg 1) m(B_{+}(r)) =:\eta' m(B_{+}(r)).
$$
Set $g=f \circ \Psi_{a} \in H^{1}(B_{+}(r))$. Using Lemma~\ref{lem:moser} with $N_1$ and $\eta'$, we have
\begin{equation*}
\left( \int_{B_{+}(r)}|g|^{q}\,dx\right)^{1/q}\le c_{1}(d,\eta')\left\{\left( r^{2-d}\int_{B_{+}(r)}|\nabla g|^{2}\,dx \right)^{1/2}+\left( r^{-d}\int_{N_1}|g|^{2}\,dx\right)^{1/2} \right\}.
\end{equation*}
On account of the changes of variable formula, \eqref{eq:jacobian2} and the inequality $x^{1/2}+y^{1/2} \le (2x+2y)^{1/2}$ $(x,y \ge 0)$, the proof is complete
for $c_{2}=\sqrt{2}c_{1}(d,\eta')M^{d\left(\frac{1}{q}+\frac{1}{2} \right)}_{n}$.
\qed\end{proof}

Let $E$ be an open subset of $D$. Following \cite{FT,FT0}, we define two subspaces of $H^{1}(E)$:
\begin{align*}
\widehat{C}(E)&=\{f \in C^{1}(E) \mid \|f\|_{H^{1}(E)}<\infty, f=0 \text{ on } \partial E \cap D \},\\
\widehat{H}(E)&=\text{the completion of }\widehat{C}(E)\text{ with respect to the norm } \|\cdot\|_{H^{1}(E)}.
\end{align*}
$\widehat{H}(E)$ is regarded as a subspace of $ H^{1}(D)$ by a natural way. Clearly, if $\ob{E} \subset D$, then $\widehat{H}(E)=H_{0}^{1}(E)$. We define the Dirichlet form $(\cE_{E},D(\cE_{E}) )$ on $L^{2}(E,m)$ by
\begin{align*}
\cE_{E}(f,g)&:=\frac{1}{2}\int_{E}(\nabla f, \nabla g)\,dx,\quad f,g \in D(\cE_{E}) :=\widehat{H}^{1}(E).
\end{align*}
For $\alpha \ge 0$ and $f,g \in \widehat{H}(E)$, we denote $\cE_{E,\alpha}(f,g)=\cE_{E}(f,g)+\alpha (f,g)_{L^{2}(E,m)}$. In the sequel, $\{G_{E,\alpha}\}_{\alpha>0}$ denotes the resolvent on $L^{2}(E,m)$ associated with $(\cE_{E},D(\cE_{E}) )$.

\begin{lemma}\label{lem:localsobolev}
For any $q \in [2,2d/(d-2)]$ ($q \in [2,\infty)$ if $d=2$), there exists $c_{3}=c_{3}(d, M_n, q)>0$ such that 
\begin{equation}
\|f\|_{L^{q}(B_{a}^{\ast}(r),m)}
\le c_3 \| \nabla f\|_{L^{2}(B_{a}^{\ast}(r),m)} \label{eq:localsobolev}
\end{equation}
for all $f \in \widehat{H}(B_{a}^{\ast}(r))$ with $a \in K_n$ and $r \in (0, 1/(2M_{n}^2 \vee 1)]$.
\end{lemma}
\begin{proof}
For each $f \in \widehat{H}(B_{a}^{\ast}(r))$ with $a \in K_n$, $r \in (0, 1/(2M_{n}^2 \vee 1)]$, define $\hat{f} \in H^{1}(B_{a}^{\ast}(1))$ by
\begin{align*}
\begin{cases}
 \hat{f}=f \text{ on } B_{a}^{\ast}(r),\\
 \hat{f}=0 \text{ on } B_{a}^{\ast}(1)\setminus B_{a}^{\ast}(r). 
 \end{cases}
 \end{align*}
Iy follows from \eqref{eq:jacobian2} that
\begin{align*}
\frac{m(B_{a}^{\ast}(r))}{m( B_{a}^{\ast}(1))} &\le \frac{M(a)^{d}m( B_{+}(1/(2M_{n}^2 \vee 1)))}{M(a)^{-d}m( B_{+}(1))} \\
& = M(a)^{2d} \times (2M_{n}^2 \vee 1)^{-d} \\
&\le  M_{n}^{2d} \times (2M_{n}^2 \vee 1)^{-d} \\
&\le 2^{-d}.
\end{align*}
Hence, $N=B_{a}^{\ast}(1) \setminus B_{a}^{\ast}(r)$ satisfies $m(N) \ge (1-2^{-d})m(B_{a}^{\ast}(1))$.
From Lemma~\ref{lem:moser2},
\begin{align*}
\|\hat{f} \|_{L^{q}(B_{a}^{\ast}(1),m)} &\le c_3 \left(  \|\nabla \hat{f} \|_{L^{2}(B_{a}^{\ast}(1),m)}^{2} + \| \hat{f} \|_{L^{2}(B_{a}^{\ast}(1) \setminus B_{a}^{\ast}(r),m)}^{2}  \right)^{1/2},
\end{align*}
where $c_{3}=c_{2}(d,M_n,1-2^{-d},q)$. By the definition of $\hat{f}$, we complete the proof.
\qed\end{proof}

From Lemma~\ref{lem:localsobolev}, $\widehat{H}({B_{a}^{\ast}(r)})$ is a Hilbert space with inner product $\cE_{B_{a}^{\ast}(r)}(\cdot,\cdot)$. Therefore, for all $f \in L^{2}(B_{a}^{\ast}(r))$,
there exists a unique element $G_{B_{a}^{\ast}(r),0}f \in \widehat{H}(B_{a}^{\ast}(r))$ such that 
\begin{equation} 
 \cE_{B_{a}^{\ast}(r)}\bigl(G_{B_{a}^{\ast}(r),0}f,g\bigr)=\int_{B_{a}^{\ast}(r)}fg\,dx,\quad  g \in \widehat{H}(B_{a}^{\ast}(r))\label{eq:func3}.
 \end{equation}
 Using \eqref{eq:localsobolev}, we obtain the following by the same argument as in \cite[Theorem~4.1]{ST}.
\begin{lemma} \label{lem:st2}
Let $p>d$. Then, there exists $c_4=c_4(d,M_{n},p)>0$ such that  
\begin{equation*}
\|G_{B_{a}^{\ast}(r),\alpha}f \|_{L^{\infty}(B_{a}^{\ast}(r),m)} \le c_{4} r^{(p-d)/p}\|f\|_{L^{p}(B_{a}^{\ast}(r),m)}
\end{equation*}
for all $\alpha \ge0$, $f \in L^{2}(\ob{D},m) \cap L^{p}(\ob{D},m)$, $a \in K_n$, $r \in (0,1/(2M_{n}^2 \vee 1)]$.
\end{lemma}

\subsection{Some estimates of solutions and subsolutions}
Let $a \in \partial D$ and $r \in (0,1]$. We say a function $f \in H^1(B_{a}^{\ast}(r))$ is a subsolution of $(\cA_{B_{a}^{\ast}(r)}^0, \widehat{H}(B_{a}^{\ast}(r))$ if for any $g \in \widehat{H}(B_{a}^{\ast}(r))$ with $g \ge 0$,
\begin{equation}
\cE_{B_{a}^{\ast}(r)}(f,g) \le0 \label{eq:eqsub}.
\end{equation}

\begin{lemma} \label{lem:cutoff}
Let $a \in \partial D$ and $r \in (0,1]$. For any positive numbers $s_1,s_2$ with $0<s_1<s_2 \le r$, there is a smooth function $\xi=\xi^{a,s_1,s_2} \in \widehat{H}(B_{a}^{\ast}(r))$ such that $\xi=1$ on $B_{a}^{\ast}(s_1) $ and $0 \le \xi \le 1$, and $|\nabla \xi | \le M(a)(s_2-s_1)^{-1}$ on $B_{a}^{\ast}(r)$.
\end{lemma}

\begin{proof}
Take a smooth function $\zeta \in C^{\infty}_{c}(\R^d)$ satisfying the following:
\begin{itemize}
\item $\zeta=1$ on $B(s_1)$, $\zeta=0$ on $\R^d \setminus B(s_2)$,
\item $|\nabla \zeta| \le (s_2-s_1)^{-1}$ on $B(r)$.
\end{itemize}
Then, $\xi^{a,s_1,s_2}:=\zeta \circ \Psi_{a}^{-1}|_{B_{a}^{\ast}(r)}$ satisfy the required conditions.
\qed\end{proof}
In the following, for an open subset $E$ of $\R^d$, we denote by $\esssup_{E}$ and $\essinf_{E}$ the essential supremum and essential infimum with respect to $m$, respectively.
\begin{lemma} \label{lem:localbound}
Let $a \in K_n$, $r \in (0, 1/(2M_{n}^2 \vee 1)]$ and $f \in H^1(B_{a}^{\ast}(r))$ be a nonnegative subsolution of \eqref{eq:eqsub}. Then, for all $0<s<r$, we have $f \in L^{\infty}(B_{a}^{\ast}(s),m)$ and there exists $c_5=c_{5}(d,M_{n})>0$ such that
\begin{equation*}
\esssup_{B_{a}^{\ast}(s)}f \le c_{5}(r-s)^{-d/2} \|f\|_{L^{2}(B_{a}^{\ast}(r),m)}.
\end{equation*}
\end{lemma}
\begin{proof}
For each $k \ge 0$, we set $f_k=(f-k)\vee 0$. Fix positive numbers $s,t$ with $0<s<t\le r$ and take $\xi=\xi^{a,s,t}$ in Lemma~\ref{lem:cutoff}.
Noting that  $f_k \xi^{2} \in  \widehat{H}(B_{a}^{\ast}(t)) \subset \widehat{H}(B_{a}^{\ast}(r))$, we have
 \begin{align*}
0 \ge \int_{B_{a}^{\ast}(r)} \bigl(\nabla f, \nabla (f_{k}\xi^{2}) \bigr)\,dx 
&=\int_{B_{a}^{\ast}(t)} |\nabla f_k|^{2}\xi^2\,dx+2\int_{B_{a}^{\ast}(t)} \xi f_k ( \nabla f_k, \nabla \xi)\,dx \\
&\ge \frac{1}{2}\int_{B_{a}^{\ast}(t)} |\nabla f_k|^{2}\xi^2\,dx-2\int_{B_{a}^{\ast}(t)} f_k^{2}|\nabla \xi|^{2}\,dx,
\end{align*}
which implies $\int_{B_{a}^{\ast}(t)} |\nabla f_k|^{2}\xi^2\,dx\le 4\int_{B_{a}^{\ast}(t)} f_k^{2}|\nabla \xi|^{2}\,dx$
 and
\begin{align}
 \| \nabla (f_k\xi) \|^2_{L^{2}(B_{a}^{\ast}(t),m)}
& \le 10 \|f_k \nabla \xi\|^{2}_{L^{2}(B_{a}^{\ast}(t),m)} \label{eq:ite1}.
\end{align}
We fix a positive number $q>2$ and define $q_d>0$ by
\begin{align*}
\begin{cases}
q_{d}=q/2\quad \text{if }d=2,\\
q_{d}=d/2 \quad  \text{if }d\ge 3.
\end{cases}
\end{align*}
Let $p_d>1$ be the positive number such that $1/p_d+1/q_d=1$.
We note that $f_k \xi \in \widehat{H}(B_{a}^{\ast}(t))$. Using Lemma~\ref{lem:localsobolev} and \eqref{eq:ite1}, we have
\begin{align}\label{eq:eqiterationinq}
&\|f_k \xi \|^{2}_{L^{2}(B_{a}^{\ast}(t),m)}  \\
&\le  \|f_k \xi \|^{2}_{L^{2p_{d}}(B_{a}^{\ast}(t),m)} \times m\left( \left\{x \in B_{a}^{\ast}(t) \mid f_k(x)\neq 0 \right\}\right)^{1/q_{d}} \notag  \\
&\le c_3(d,M_n,2p_{d}))^{2} \| \nabla (f_k\xi) \|^2_{L^{2}(B_{a}^{\ast}(t),m)} \times m\left( \left\{x \in B_{a}^{\ast}(t) \mid f_k(x)\neq 0 \right\}\right)^{1/q_{d}} \notag \\
&  \le 10c_3^{2}  \|f_k \nabla \xi\|^{2}_{L^{2}(B_{a}^{\ast}(t),m)} \times m\left( \left\{x \in B_{a}^{\ast}(t) \mid f_k(x)\neq 0 \right\}\right)^{1/q_{d}}. \notag
\end{align}
From the definition of $\xi$ and \eqref{eq:eqiterationinq}, we have
\begin{equation}
 \|f_k\|_{L^{2}(A_{k,s},m) }^2  \le 10c_3^{2}M^2_{n}(t-s)^{-2}\|f_k\|_{L^{2}(A_{k,t},m) }^2 \times m( A_{k,t})^{1/q_{d}}
\label{eq:ite2},
\end{equation}
where we define $A_{k,s}=\{ x \in B_{a}^{\ast}(s) \mid f(x) > k\}$. 
Fix $l>k \ge 0$. It follows from \eqref{eq:ite2} and  Chebyshev's inequality that 
\begin{align}\label{eq:ite3}
 \|f_l\|_{L^{2}(A_{l,s},m) }^2  &\le 10c_3^{2}M^2_{n}(t-s)^{-2}\|f_l\|_{L^{2}(A_{l,t},m) }^2 \times m( A_{l,t})^{1/q_{d}} \\
& \le 10c_3^{2}M^2_{n}(t-s)^{-2}\|f_k\|_{L^{2}(A_{k,t},m) }^2 \times \{ (l-k)^{-2}\|f-k\|_{L^{2}(A_{k,t},m)}^2\}^{1/q_{d}} \notag \\
&=10c_3^{2}M^2_{n}(t-s)^{-2}(l-k)^{-2/q_{d}}\|f_k\|_{L^{2}(A_{k,t},m) }^{2+2/q_d}  \notag.
\end{align}
We note that \eqref{eq:ite3} holds with any $k \ge0 $ and $s,t$ with $0<s< t\le r$. 
We define $s_\nu =s+(r-s)/2^{\nu}$ for $\nu \in \N \cup \{0\}$. For some $K>0$ to be determined, we also define 
\begin{equation*}
k_\nu=K(1-1/2^{ \nu }),\quad  \nu \in \N \cup \{0\}.
\end{equation*}
It is easy to see that $k_0=0, s_0=r$ and $k_{\nu}-k_{\nu-1}=K/2^\nu$, $s_{\nu-1}-s_{\nu}=(r-s)/2^\nu$.  For each $k\ge 0$ and $s \in (0,r)$, we define $\varphi(k,s)=\|f_k\|_{L^{2}(A_{k,s},m) }$. It follows from \eqref{eq:ite3} that
\begin{align}\label{eq:ite-1}
\varphi(k_{\nu},s_{\nu}) &\le \sqrt{10}c_3M_{n}(s_{\nu-1}-s_{\nu})^{-1}(k_{\nu}-k_{\nu-1})^{-1/q_d} \varphi(k_{\nu-1},s_{\nu-1})^{1+1/q_{d}} \\
& =\sqrt{10}c_3M_n(r-s)^{-1}K^{-1/q_d}2^{\nu(1+1/q_{d})}\varphi(k_{\nu-1},s_{\nu-1})^{1+1/q_{d}} \notag .
\end{align}
Next, we will prove that there exists $\gamma >1$ such that for $\nu \in \N \cup \{0\}$,
\begin{equation}
\varphi(k_\nu,s_\nu) \le \varphi(k_{0},s_{0})\gamma^{-\nu} \label{eq:ite4}.
\end{equation}
Obviously, \eqref{eq:ite4} is true for $\nu=0$. Assume \eqref{eq:ite4} is true for $\nu-1$. Using \eqref{eq:ite-1}, we have
\begin{align*}
\varphi(k_{\nu},s_{\nu}) &\le \sqrt{10}c_3M_n(r-s)^{-1}K^{-1/q_d}2^{\nu(1+1/q_d)}\varphi(k_{\nu-1},s_{\nu-1})^{1+1/q_d} \\
&\le \sqrt{10}c_3M_n(r-s)^{-1}K^{-1/q_d}2^{\nu(1+1/q_d)}\{ \varphi(k_{0},s_{0})\gamma^{-(\nu-1)}\}^{1+1/q_d}\\
&=\sqrt{10}c_3M_n(r-s)^{-1}K^{-1/q_d} \varphi(k_0,s_0)^{1/q_d} \gamma^{1+1/q_d} 2^{\nu(1+1/q_d)} \gamma^{-\nu/q_d} \varphi(k_0,s_0) \gamma^{-\nu}.
\end{align*}
Let $\gm=2^{1+q_d}$. Then, $2^{\nu(1+1/q_d)} \gamma^{-\nu/q_d}=1$. We choose $K>0$ such that
\[
\sqrt{10}c_3M_n(r-s)^{-1}K^{-1/q_d} \varphi(k_0,s_0)^{1/q_d} \gamma^{1+1/q_d}  = 1.
\]
Then, \eqref{eq:ite4} holds for $\nu \in \mathbb{N} \cup \{0\}$. Letting $\nu \to \infty$ in \eqref{eq:ite4}, we have $(f-K)\vee 0=0$ $m$-a.e. on $B_{a}^{\ast}(s)$, which implies
\begin{align*}
\esssup_{B_{a}^{\ast}(s)} f &\le10^{q_d/2}c_3^{q_d}M_{n}^{q_d}(r-s)^{-q_d} \varphi(k_0,s_0) 2^{(1+q_d)^2} \\
& =10^{q_d/2}c_3^{q_d}M_{n}^{q_d} 2^{(1+q_d)^2}(r-s)^{-q_d}  \|f\|_{L^{2}(B_{a}^{\ast}(r),m)}.
\end{align*}
If $d=2$, letting $q \to 2$, we have the claim.
\qed\end{proof}

\begin{lemma}\label{lem:solsol}
Let $a \in \partial D$, $r \in (0,1]$. If $f \in H^{1}(B_{a}^{\ast}(r))$ satisfies
\begin{equation}
\cE_{B_{a}^{\ast}(r),\alpha}(f,g)=0,\qquad g \in \widehat{H}(B_{a}^{\ast}(r)) \label{eq:sol}
\end{equation}
for some $\alpha \ge 0$, then $ f \vee 0$ and $(-f) \vee 0$ are  nonnegative subsolutions of \eqref{eq:eqsub}. 
\end{lemma}
\begin{proof}
Let $\{ \psi_{\eps}\}_{\eps>0}$ be convex smooth functions on $\R$ such that each $\psi_{\eps}'(x)$ and $\psi_{\eps}''(x)$ are bounded and $\lim_{\eps \to 0}\psi_{\eps}(x)= x \vee 0$, and $\lim_{\eps \to 0}\psi_{\eps}'(x) =\bone_{ [0,\infty)}(x)$. Such $\{\psi_\eps\}_{\eps>0}$ can be constructed by mollifying the function $x \mapsto x \vee0 $. Take $g \in \widehat{H}(B_{a}^{\ast}(r)) $ with $g \ge 0$. We may assume $g \in L^{\infty}(B_{a}^{\ast}(r),m )$. Since $\psi'_{\eps}(f)g \in \widehat{H}(B_{a}^{\ast}(r))$, we have
\begin{align*}
\cE_{B_{a}^{\ast}(r)}\bigl (\psi_{\eps}(f), g\bigr) 
&=\frac{1}{2}\int_{B_{a}^{\ast}(r)}\bigl( \nabla f, \nabla(\psi'_{\eps}(f)g) \bigr)\,dx-\frac{1}{2}\int_{B_{a}^{\ast}(r)}g\psi''_{\eps}(f) |\nabla f|^2\,dx\\
&\le \frac{1}{2}\int_{B_{a}^{\ast}(r)}\bigl( \nabla f, \nabla(\psi'_{\eps}(f)g) \bigr)\,dx.
\end{align*}
Here, we used the convexity of $\psi_\eps$. Hence, we obtain
\begin{align*}
&\cE_{B_{a}^{\ast}(r)}\bigl (\psi_{\eps}(f), g\bigr)  \le \frac{1}{2}\int_{B_{a}^{\ast}(r)}\bigl( \nabla f, \nabla(\psi'_{\eps}(f)g) \bigr)\,dx \\
&=\cE_{B_{a}^{\ast}(r),\alpha}(f,\psi'_{\eps}(f)g) -\alpha \int_{B_{a}^{\ast}(r)} f \psi'_{\eps}(f)g\,dx =0-\alpha \int_{B_{a}^{\ast}(r)} f \psi'_{\eps}(f)g\,dx.
\end{align*}
Letting $\eps \to 0$, we obtain the claim.
We can similarly prove $(-f) \vee 0$ is a nonnegative subsolution of \eqref{eq:eqsub}.
\qed\end{proof}

Using Lemma~\ref{lem:localbound} and Lemma~\ref{lem:solsol}, we obtain the following corollary.
\begin{corollary}\label{lem:sol2}
Let $a \in K_n$ and $r \in (0,1/(2M_{n}^2 \vee 1)]$. Every solution $f \in H^1(B_{a}^{\ast}(r))$
of \eqref{eq:sol} satisfies
\begin{equation*}
\esssup_{B_{a}^{\ast}(s)}|f| \le c_5 (r-s)^{-d/2}\|f\|_{L^{2}(B_{a}^{\ast}(r),m)},\quad 0<s<r.
\end{equation*}
In paticular, $f$ is bounded on $B_{a}^{\ast}(s)$.
\end{corollary}

For each $\eps\in(0,1)$, we define $F_{\eps}(x)=(-\log(x+\eps)) \vee 0 $.
\begin{lemma} \label{lem:lemb}
Let $a \in K_n$ and $r \in (0,1/(2M_{n}^2 \vee 1)]$. Then, there exists $c_6=c_6(d,M_n)>0$ such that for any  nonnegative solution $f \in H^1
(B_{a}^{\ast}(r))$ of \eqref{eq:sol} with $\alpha=0$,
\begin{equation*}
\sup_{\eps \in (0,1)}\|\nabla F_{\eps}(f)\|_{L^{2}(B_{a}^{\ast}(r/2),m)}^2
\le c_{6}r^{d-2}.
\end{equation*}
\end{lemma}
\begin{proof}
We write $E$ for $B_{a}^{\ast}(r)$. Using Friedrich's mollifier technique, for each $\eps \in (0,1)$, we can take smooth functions $\{F_{\eps,\dl}\}_{\dl>0}$ on $[0,\infty)$ such that each derivative $F_{\eps,\delta}'$ is Lipschitz continuous and  $F_{\eps,\dl}'' \ge (F_{\eps,\dl}')^2$, and $\lim_{\dl \to 0}F_{\eps,\dl}(f)=F_{\eps}(f)$ in $L^{2}(E,m)$. Take $\xi^{a,r/2,r} \in C^{\infty}(E)$ in Lemma~\ref{lem:cutoff}. Then, $F_{\eps,\dl}'(f)\xi^2 \in \widehat{H}(E)$. Since $f$ is a solution of \eqref{eq:sol},
\begin{align*}
0&=\cE_{E}(f,F_{\eps,\dl}'(f)\xi^2)=\frac{1}{2}\int_{E} \bigl(\nabla F_{\eps,\dl}(f), \nabla \xi^2 \bigr)\,dx+\frac{1}{2}\int_{E}\xi^2 F_{\eps,\dl}''(f) |\nabla f|^{2}\,dx \notag \\
&\ge \int_{E} \xi \bigl(\nabla F_{\eps,\dl}(f), \nabla \xi \bigr)\,dx+\frac{1}{2}\int_{E} |\xi \nabla F_{\eps,\dl}(f)|^{2}\,dx  \notag \\
&\ge -\left( \int_{E} |\xi \nabla F_{\eps,\dl}(f)|^2\,dx \right)^{1/2} \left(\int_{E} |\nabla \xi|^2\,dx\right)^{1/2}+\frac{1}{2}\int_{E} |\xi \nabla F_{\eps,\dl}(f)|^{2}\,dx.
\end{align*}
This inequality leads to
\begin{equation*}
\int_{E} \xi^2 |\nabla F_{\eps,\dl}(f)|^{2}\,dx\le 4\int_{E} |\nabla \xi|^{2}\,dx.
\end{equation*}
By the definition of $\xi$, we obtain
\begin{equation*}
\|\nabla F_{\eps, \dl}(f)\|_{L^{2}(B_{a}^{\ast}(r/2),m)}^2
\le c_{6}r^{d-2}.
\end{equation*}
with $c_{6}(d,M_n)=16m(B_{+}(1))M^{d+2}_{n}$. Letting $\dl \to 0$, we obtain this lemma.
\qed\end{proof}

\begin{lemma}\label{lem:lower}
Let $a \in  K_n$, $\kappa \in (0,1]$, $r \in (0,1/(2M_{n}^2 \vee 1)]$, and let $f \in H^1(B_{a}^{\ast}(r))$ be a nonnegative solution of \eqref{eq:sol} with $\alpha=0$. If $f  \in H^1(B_{a}^{\ast}(r))$ satisfies 
\begin{equation}
m( \{ f \ge 1\} \cap B_{a}^{\ast}(r/2)) \ge \kappa m( B_{a}^{\ast}(r/2)) \label{eq:cond}, 
\end{equation}
then there exists $c_7=c_{7}(d,M_n,\kappa) \in (0,1]$ such that
\begin{equation*}
\essinf_{B_{a}^{\ast}(r/4)} f \ge c_7.
\end{equation*}
\end{lemma}
\begin{proof}
Fix $\eps \in (0,1)$. Since $F_{\eps}$ is a Lipschitz continuous function on $[0,\infty)$, we have
$F_{\eps}(f) \in H^1( B_{a}^{\ast}(r))$. Since this is a nonnegative subsolution of \eqref{eq:eqsub}, we see from Lemma~\ref{lem:localbound} that
\begin{equation}
\esssup_{B_{a}^{\ast}(r/4)}F_{\eps}(f)
 \le c_5(d,M_n)\times (r/4)^{-d/2} \|F_{\eps}(f) \|_{L^{2}(B_{a}^{\ast}(r/2),m)} \label{eq:harnack1}.
\end{equation}
By Lemma~\ref{lem:moser2} with $q=2$ and $N_2=\{ f \ge 1\} \cap B_{a}^{\ast}(r/2)$) and Lemma~\ref{lem:lemb}, the right hand side of \eqref{eq:harnack1} is estimated as follows:
\begin{align}\label{eq:harnack2}
&c_5(d,M_n) \times (r/4)^{-d/2} \|F_{\eps}(f) \|_{L^{2}(B_{a}^{\ast}(r/2),m)} \\
  &\le c_5 \times (r/4)^{-d/2} \left\{ c_2(d,M_n,\kappa, 2)^{2}  r^{2} \left( \|\nabla F_{\eps}(f)\|_{L^{2}(B_{a}^{\ast}(r/2),m)}^2+\| F_{\eps}(f)\|_{L^{2}(N_2,m)}^2\right) \right\}^{1/2} \notag \\
&=4^{d/2} c_{2}c_5 r^{-d/2+1}  \|\nabla F_{\eps}(f)\|_{L^{2}(B_{a}^{\ast}(r/2),m)} \notag \\
&\le 4^{d/2} c_{2}c_{5}\times c_{6}(d,M_n)^{1/2} \notag.
\end{align}
From \eqref{eq:harnack1}, \eqref{eq:harnack2} and the definition of $F_{\eps}$, 
$$f(x)+\eps \ge c_7 \text{ $m$-a.e. $x \in B_{a}^{\ast}(r/4)$},$$
where $c_7=\exp(-4^{d/2} c_2c_5 c_6^{1/2}) \in (0,1]$.
Since $\eps \in (0,1)$ is arbitrary, we obtain the lemma.
\qed\end{proof}

In the sequel, for a bounded function $f$ defined on an open set $E$ of $D$, we define
\[
\text{\rm Osc}(f;E)=\esssup_{E}f-\essinf_{E}f.
\]
\begin{lemma} \label{lem:st1}
Let $a \in K_n$ and $r \in (0,1/(2M_{n}^2 \vee 1)]$. Let $w \in H^1(B_{a}^{\ast}(r))$ be a solution of \eqref{eq:sol} with $E=B_{a}^{\ast}(r)$ and $\alpha = 0$. Futhermore, we assume $w \in L^{\infty}(B_{a}^{\ast}(r),m)$. Then, there exists $c_{8}=c_{8}(d,M_n) \in [1/2,1)$ such that
\begin{equation*}
\text{\rm Osc}(w;B_{a}^{\ast}(s/4)) \le c_{8}\text{\rm Osc}(w;B_{a}^{\ast}(s)),\quad 0<s \le r.
\end{equation*}
\end{lemma}
\begin{proof}
Note that its oscillation does not change by additing a constant to $w$. Therefore, by additing an appropriate constant to $w$, we can assume
\begin{equation*}
\esssup_{B_{a}^{\ast}(s)}w=-\essinf_{B_{a}^{\ast}(s)}w=\frac{1}{2}\text{Osc}(w;B_{a}^{\ast}(s))=:K.
\end{equation*}
Then,
$(K+w)/K,(K-w)/K$ satisfy \eqref{eq:sol} with $\alpha=0$ and $E=B_{a}^{\ast}(s)$. Indeed, since $g \in \widehat{H}(B_{a}^{\ast}(s)) \subset  \widehat{H}(B_{a}^{\ast}(r))$ and $g=0$ $m$-a.e. on $B_{a}^{\ast}(r) \setminus B_{a}^{\ast}(s)$, it holds that
\begin{align*}
0=\cE_{B_{a}^{\ast}(r),0}(w,g) &=\frac{1}{2}\int_{B_{a}^{\ast}(r)\setminus B_{a}^{\ast}(s)}(\nabla w,\nabla g)\,dm+\frac{1}{2}\int_{B_{a}^{\ast}(s)}(\nabla w,\nabla g)\,dm\\
&=0+\cE_{B_{a}^{\ast}(s),0}(w,g). 
\end{align*}
By using this equality and noting that $\nabla \bone_{D}=0$, we obtain
\begin{align*}
0&=\frac{1}{2}\int_{B_{a}^{\ast}(s)}(\nabla \bone_{D},\nabla g)\,dm\pm\frac{1}{K}\cE_{B_{a}^{\ast}(s),0}(w,g) =\cE_{B_{a}^{\ast}(s),0}\left(\frac{K \pm w}{K},g \right).
\end{align*}
$(K+w)/K,(K-w)/K$ are both nonnegative and at least one of them satisfies \eqref{eq:cond} with $\kappa=1/2$ and $r=s$. Indeed, if both of them don't satisfy \eqref{eq:cond},
\begin{align*}
&m( B_{a}^{\ast}(s/2))=(1/2)m( B_{a}^{\ast}(s/2))+(1/2)m( B_{a}^{\ast}(s/2)) \\
&>m( \{ (K+w)/K \ge 1\} \cap B_{a}^{\ast}(s/2))+m( \{ (K-w)/K \ge 1\} \cap B_{a}^{\ast}(s/2))\\
&=m( \{w \ge 0\} \cap B_{a}^{\ast}(s/2))+m( \{ w \le 0\} \cap B_{a}^{\ast}(s/2))  \\
&\ge m( B_{a}^{\ast}(s/2)),
\end{align*}
which is a contradiction. If $(K+w)/K$  satisfies \eqref{eq:cond} for $\kappa=1/2$, we have
$$\essinf_{B_{a}^{\ast}(s/4)} (K+w)/K \ge c_{9}(d,M_n,1/2) $$ from Lemma~\ref{lem:lower}.
Therefore, 
\begin{align*}
c_{7}K-K \le w \le K \text{ $m$-a.e. on $B_{a}^{\ast}(s/4)$.}
\end{align*}
 This implies the claim with $c_{8}=1-c_{7}(d,M_n,1/2) /2 \in [1/2,1)$. In the same manner, we can obtain the claim if $(K-w)/K$  satisfies \eqref{eq:cond} for $\kappa=1/2$.
\qed\end{proof}

\subsection{Proof of Thereom~\ref{thm:wrsf}}
\begin{lemma}\label{lem:bounded}
Let $r \in(0,1/(2M_{n}^2 \vee 1)]$, $\alpha>0$,  $p>d$, and $f \in L^{p}(\ob{D},m) \cap L^{2}(\ob{D},m)$. Then, there exists $c_{9}(d,M_n,r,\alpha,p,f)>0$ such that
\begin{equation*}
\sup_{a \in K_n}\| G_{\alpha}f \|_{L^{\infty}(B_{a}^{\ast}(r/2),m)} \le c_{9}.
\end{equation*}
\end{lemma}
\begin{proof}
Set $u=G_{\alpha}^{0}f$ and $E=B_{a}^{\ast}(r)$. Let $v=G_{E,0}(f-\alpha u) \in \widehat{H}(E)$ be the solution of the equation of \eqref{eq:func3} with $\alpha=0$ and $f=f-\alpha u$. Using Lemma~\ref{lem:st2}, we have
\begin{align}
\|v \|_{L^{\infty}(E,m)}  
&\le c_4(d,M_n,p)\cdot r^{(p-d)/p}\|f-\alpha u\|_{L^{p}(E,m)} \label{eq:osck21} \\
&\le c_{4}  r^{(p-d)/p}\left( \|f\|_{L^{p}(E,m)}+\alpha \|G_{\alpha}f\|_{L^{p}(E,m)} \right) \notag \\
&\le 2c_{4} r^{(p-d)/p} \|f\|_{L^{p}(D,m)}  \notag.
\end{align}
Setting $w=u-v$, we have, for all $g \in \widehat{H}(E)$, 
\begin{align*}
\cE_{E}(w,g)&=\cE_{E}(G_{\alpha}^{0}f-G_{E,0}(f-\alpha u),g) \\
&=\cE_{E}(G_{\alpha}^{0}f,g)-(f-\alpha u,g)_{L^{2}(D,m)} \\
&=\cE (G_{\alpha}^{0}f,g)-(f-\alpha u,g)_{L^{2}(D,m)}\\
&=\cE_{\alpha}(G_{\alpha}^{0}f,g)-(f,g)_{L^{2}(D,m)}=0.
\end{align*}
From Corollary~\ref{lem:sol2} and \eqref{eq:osck21}, we have 
\begin{align}
&\|w\|_{L^{\infty}(B_{a}^{\ast}(r/2))}  \label{eq:osck3} \\
&\le c_5(d,M_n) \times (r/2)^{-d/2}\|w\|_{L^{2}(B_{a}^{\ast}(r),m)} \notag \\
&\le c_5 \times (r/2)^{-d/2}(\|G_{\alpha}^{0}f\|_{L^{2}(B_{a}^{\ast}(r),m)}+\|v\|_{L^{2}(B_{a}^{\ast}(r),m)}) \notag \\
&\le c_5 \times (r/2)^{-d/2}(\alpha^{-1}\|f\|_{L^{2}(D,m)}+2c_{4} r^{(p-d)/p} \|f\|_{L^{p}(D,m)} \times M_{n}^{d/2} m(B_{+}(1))^{1/2} ) \notag.
\end{align}
Using \eqref{eq:osck21}, \eqref{eq:osck3}, and the relation $u=w+v$, we complete the proof.
\qed\end{proof}
\begin{lemma}\label{lem:osc}
Let $\alpha>0$, $p>d$, $f \in L^{p}(\ob{D},m) \cap L^{2}(\ob{D},m)$, and $r \in (0,1/(2M_{n}^2 \vee 1)]$. There exist $c_{11}=c_{11}(d,M_n,p,\alpha,f,r)>0$, $q_{1}=q_{1}(d,M_n,p)\in (0,1)$ such that
\begin{equation*}
\sup_{a \in K_n}\text{\rm Osc}(G_{\alpha}^{0}f;B_{a}^{\ast}(s)) \le c_{11}s^{q_1},\quad 0<s \le r/8.
\end{equation*}
\end{lemma}
\begin{proof}
Set $u=G_{\alpha}^{0}f$, $s \in (0,r/2]$, and $F=B_{a}^{\ast}(s)$. Let $v=G_{F,0}(f-\alpha u) \in \widehat{H}(F)$ be the solution of the equation of \eqref{eq:func3} with $\alpha=0$ and $f=f-\alpha u$. Using Lemma~\ref{lem:st2}, we have
\begin{align}\label{eq:osck2}
\|v \|_{L^{\infty}(B_{a}^{\ast}(s),m)} &\le c_4(d,M_n,p)s^{(p-d)/p}\|f-\alpha u\|_{L^{p}(B_{a}^{\ast}(s),m)} \\
&\le c_{4}  s^{(p-d)/p}\left( \|f\|_{L^{p}(B_{a}^{\ast}(s),m)}+\alpha \|G_{\alpha}f\|_{L^{p}(B_{a}^{\ast}(s),m)} \right) \notag \\
&\le 2c_{4} s^{(p-d)/p} \|f\|_{L^{p}(D,m)} .\notag
\end{align}
Setting $w=u-v$, we can check $w \in H^{1}(F)$ satisfies \eqref{eq:sol} with $\alpha=0$ by the same argument as in the proof of Lemma~\ref{lem:bounded}. Therefore,
using  Lemma~\ref{lem:st1} and Lemma~\ref{lem:bounded}, we have
$$
\text{Osc}(w;B_{a}^{\ast}(s/4)) \le c_{8}(d,M_n)  \text{Osc}(w;B_{a}^{\ast}(s)).
$$
Hence, for all $s \in (0,r/2]$,
\begin{align*}
 \text{Osc}(u;B_{a}^{\ast}(s/4))&\le \text{Osc}(w;B_{a}^{\ast}(s/4))+\text{Osc}(v;B_{a}^{\ast}(s/4))\\
&\le c_{8} \text{Osc}(w;B_{a}^{\ast}(s))+2\|v\|_{L^{\infty}(F,m)}  \\
&\le  c_{8} \text{Osc}(u;B_{a}^{\ast}(s))+4\|v\|_{L^{\infty}(F,m)}  \\
&\le c_{8} \text{Osc}(u;B_{a}^{\ast}(s))+8c_{4} s^{(p-d)/p} \|f\|_{L^{p}(D,m)}.
\end{align*}
In the last inequality we used \eqref{eq:osck2}.
Using \cite[Lemma~7.3]{ST}, we obtain 
\begin{align*}
\text{Osc}(u;B_{a}^{\ast}(s/4)) \le \left\{\frac{8\text{Osc}(u;B_{a}^{\ast}(r/8))}{r}+\frac{8c_4 \|f\|_{L^{p}(D,m)}}{1-c_{10}} \right\}s^{q_1},\quad 0<s \le r/2. 
\end{align*}
Here, $c_{10}=c_{10}(c_8)\in (0,1)$ and $q_{1}=q_{1}(c_{8},(p-d)/p)>0$. Using Lemma~\ref{lem:bounded}, we have
\begin{align*}
\text{Osc}(u;B_{a}^{\ast}(s)) \le \left\{\frac{4^{q_1+2}\cdot c_{9}(d,M_n,p,\alpha,f,r)}{r}+\frac{4^{q_1}\cdot8c_4 \|f\|_{L^{p}(D,m)}}{1-c_{10}} \right\}s^{q_1},\quad 0<s \le r/8. 
\end{align*}
This implies the lemma.
\qed\end{proof}

Stampacchia~\cite{ST} gives an estimate for oscilications on open balls with closures contained in $D$:
\begin{lemma}\label{lem:osc2}
Let $\eta>0$, $a \in D \setminus D^\eta$ and $E=B(a,\eta)$, where $D^{\eta}=\{x \in D \mid \text{dist}(x,\partial D) <\eta\}$. Let $\alpha>0$, $p>d$, and $f \in L^{2}(\ob{D},m) \cap L^{p}(\ob{D},m)$. Then,
 \begin{equation*}
 \text{\rm Osc}(G_{\alpha}^{0}f;B(a,s)) \le c_{12} s^{q_{2}},\quad 0<s \le \eta/4,
 \end{equation*}
where $c_{12}>0$, $q_{2} \in (0,1)$ and these constants are independent of $a \in D \setminus D^{\eta}$. $c_{12}$ may depend on $f$.
\end{lemma}

Take $\eta \in (0,1)$ and $x \in D_n \cap D^{\eta}$. By the definition of $D^{\eta}$, there exists $a \in \partial D$ such that $|x-a|<\eta$. Since $x \in D_n$, 
\begin{align*}
a &\in \{z \in \partial D \mid \text{dist}(z,D_n)<\eta\} \\
&\subset \{z \in \partial D \mid \text{dist}(z,D_n)<1\} \\
&\subset \{z \in \partial D \mid \text{dist}(z,B(n))<1\}.
\end{align*}
This implies $a \in \ob{B(n+1)}$.
Therefore, for any $x \in D_n \cap D^{\eta}$, there exists $a \in K_{n+1}$ such that $|x-a|<\eta$. 

\begin{proof}[Proof of Theorem~\ref{thm:wrsf}]
From Lemma~\ref{lem:osc}, for any $\eps>0$, there exists $s_{1}=s_{1}(\eps) \le 1/\{8(2M_{n+1}^2 \vee 1)\}$ such that
\begin{equation}
\sup_{a \in K_{n+1}}\text{Osc}(G_{\alpha}^{0}f;B_{a}^{\ast}(s_1))<\eps \label{eq:osc3}.
\end{equation} 
Set $\eta=s_{1}(\eps)/2(M_{n+1}+1).$
For any $x \in D_n \cap D^{\eta}$, there is $a \in K_{n+1} $ with $|x-a|<\eta$. It follows from $\text{Lip}(\Psi_{a}^{-1}) \le M_{n+1}$ that
\begin{align*}
x &\in B(a, s_{1}(\eps)/2(M_{n+1}+1)) \cap D \\
&\subset B(a, s_{1}(\eps)/M_{n+1}) \cap D  \\
&\subset \Psi_{a}(B(s_{1}(\eps))) \cap D=B_{a}^{\ast}(s_1(\eps)).
\end{align*}
Take $y \in D_n \cap D^{\eta}$ with $|x-y|<\eta$.
Since $x\in B(a, s_{1}(\eps)/2(M_{n+1}+1)) \cap D$,
\begin{align*}
y &\in B(a, s_{1}(\eps)/M_{n+1}) \cap D \subset B_{a}^{\ast}(s_1(\eps)).
\end{align*}
Therefore, every pair $x,y\in D_n \cap D^{\eta}$ with
 $|x-y|<\eta$ is simultaneously contained in $B_{a}^{\ast}(s_1)$ with some $a \in K_{n+1}$. Lemma~\ref{lem:osc2} with $\eta=s_{1}(\eps)/2(M_{n+1}+1)$ implies
 \begin{equation}
 \text{Osc}(G_{\alpha}^{0}f;B(a,s_2)) <\eps,\quad a \in D_n \setminus D^{\eta/2}\label{eq:osc4},
 \end{equation}
 for some $s_{2}=s_{2}(\eps,\eta) \in (0,\eta/8]$.
 
 We now set $\rho =(\eta/2) \wedge s_{2}$. Let $x,y \in D_n$ with $|x-y|<\rho $. If $x$ or $y$ belongs to $D^{\eta/2}$, then $x,y \in D_{n} \cap D^{\eta}$ and $|G_{\alpha}^{0}f
 (x)-G_{\alpha}^{0}f(y)|<\eps$ by \eqref{eq:osc3}. Otherwise, $|G_{\alpha}^{0}f
 (x)-G_{\alpha}^{0}f(y)|<\eps$ by \eqref{eq:osc4}.
\qed\end{proof}

\section{Proof of Proposition~\ref{prop:1}}
\subsection{Relative capacities and extension domains}
In this section, we introduce the definitions of relative capacity and extension domains. We also give some results on them. They are used for the proof of Proposition~\ref{prop:1} and Thereom~\ref{thm:doubf}.
\begin{definition}[{%
 \cite[Definition~5.1]{B}}]\label{defn:relative}
\begin{enumerate}
\item Let $E$ be an open subset of $\R^d$. For $A \subset \ob{E}$, we define
\begin{equation*} 
\Capa_{\ob{E}}(A)=\inf \{ \|f\|_{H^{1}(E)}^2 \mid f \in Y_{E}(A) \},
\end{equation*}
where $Y_{E}(A)$ consists of all functions $f \in \widetilde{H}^{1}(E)$ such that $f \ge 1$ $m$-a.e. on $O \cap E$ for an open subset $O \subset \R^d$ with $A \subset O$.
If $E=\R^d$, we write $\Capa(\cdot)$ for $\Capa_{\R^d}(\cdot)$. 
\item If $A \subset \ob{E}$ satisfies $\Capa_{\ob{E}}(A)=0$, $A$ is called $\Capa_{\ob{E}}$-polar set in $\ob{E}$. Moreover, if $A \subset \ob{E}$ and $\mathcal{S}(x)$ is a statement in $x \in A$, then we say that $\mathcal{S}$ holds $\Capa_{\ob{E}}$-q.e. on $A$ if $\{x \in A \mid \mathcal{S}(x) \text{ fails} \}$ is $\Capa_{\ob{E}}$-polar set. When $A=\ob{E}$, we simply say $\mathcal{S}$ holds $\Capa_{\ob{E}}$-q.e.
\item A function $f$ on $A \subset \ob{E}$ is called $\Capa_{\ob{E}}$-quasi continuous on $A$ if for each $\eps>0$ there exists an open subset $N$ of $\ob{E}$ with $\Capa_{\ob{E}}(N)<\eps$ such that $f$ is continuous on $A \setminus N$.
\end{enumerate}
\end{definition}

\begin{remark}
\begin{enumerate}
\item
Let $E$ be an open subset of $\R^d$. If $H^{1}(E)$ is a regular Dirichlet form on $L^{2}(\ob{E},\bone_{E}m)$, we have $\widetilde{H}^{1}(E)=H^{1}(E)$. 
\item One can show that $\Capa_{\ob{E}}(A)=\inf \{\Capa_{\ob{E}}(O) \mid O\text{ is open in }\ob{E} \text{ and } A \subset O \}$.
\item
If $H^{1}(E)$ is a regular Dirichlet form on $L^{2}(\ob{E},\bone_{E}m)$, we can define another capacity according as \cite[Chapter~2]{FOT}. (i), (ii) shows this capacity coincides with $\Capa_{\ob{E}}$.
\end{enumerate}
\end{remark}

\begin{definition}
An open subset $E \subset \R^d$ is called a $W^{1,2}$-extension domain if there exists a bounded linear operator $T:W^{1,2}(E) \to W^{1,2}(\R^d)$ such that $Tf=f$ $m$-a.e. on $E$.
\end{definition}

\begin{theorem}[{%
\cite[Theorem~5.2]{B}}]\label{thm:bound}
Let $E$ be an open subset of $\R^d$. For every $f \in \widetilde{H}^{1}(E)$ there is a $\text{\rm Cap}_{\ob{E}}$-quasi continuous version $\tilde{f}^{\ob{E}}:\ob{E} \to \R$ defined uniquely quasi-everywhere.
\end{theorem}

\begin{lemma} \label{lem:comp}
Let $E_1 \subset E_2 \subset \R^d$ be open sets. Then,
\begin{equation*}
\text{\rm Cap}_{\ob{E_1}}(A) \le \text{\rm Cap}_{\ob{E_2}}(A),\quad A \subset \ob{E}_1.
\end{equation*}
\end{lemma}
\begin{proof}
Let $f \in Y_{E_2}(A)$. Then, $g:=f|_{E_1}\in Y_{E_1}(A)$. Therefore, $$\Capa_{\ob{E_1}}(A)\le \|g\|_{H^{1}(E_1)} \le \|f\|_{H^{1}(E_2)}.$$
Taking infimum in $f \in Y_{E_2}(A)$, we obtain the claim.
\qed\end{proof}

\begin{proposition}\label{prop:comp}
Let $E_1 \subset E_2$ be open subsets of $\R^d$. Suppose $E_1$ is a $W^{1,2}$-extenstion domain. Then, there exists $C=C({E_1})>0$ such that 
\begin{align*}
\text{\rm Cap}_{\ob{E_1}}(A) \le \text{\rm Cap}_{\ob{E_2}}(A) \le  C \cdot  \text{\rm Cap}_{\ob{E_1}}(A),\quad A \subset \ob{E}_{1}.
\end{align*}
\end{proposition}
\begin{proof}
Since $E_1$ is a $W^{1,2}$-extension domain, we can apply
 \cite[Theorem~5.4]{B} to show that there exists $C=C(E_1)>0$ such that
 \begin{align*}
 \Capa(A) \le  C \cdot  \Capa_{\ob{E_1}}(A)
\end{align*}
for all $A \subset \ob{E}_1$. Using Lemma~\ref{lem:comp}, we have $ \Capa_{\ob{E_2}}(A) \le  C  \cdot \Capa_{\ob{E_1}}(A)$. The remaining inequality is clear from Lemma~\ref{lem:comp}.
\qed\end{proof}

\begin{proposition}\label{prop:comp2}
Let $E_1 \subset E_2$ be open subsets of $\R^d$. Suppose $E_1$ is a $W^{1,2}$-extenstion domain. 
Then, $f$ is $\text{\rm Cap}_{\ob{E_1}}$-quasi continuous on $\ob{E}_1$ if and only if  $f$ is  $\text{\rm Cap}_{\ob{E_2}}$-quasi continuous on $\ob{E}_1$.
\end{proposition}
\begin{proof}
If $f$ is $\Capa_{\ob{E_1}}$-quasi continuous on $\ob{E}_1$, for any $\eps>0$, there exists an open subset $O_{\eps}$ of $\ob{E}_1$ such that $\Capa_{\ob{E_1}}(O_{\eps})<\eps$ and $f$ is continuous on $\ob{E}_1 \setminus O_{\eps}$. Using Proposition~\ref{prop:comp},  we have
\begin{align*}
\Capa_{\ob{E_2}}(O_{\eps})\le C(E_1)\cdot \Capa_{\ob{E_1}}(O_{\eps})<C(E_1) \eps.
\end{align*}
We can take an open subset $A_{\eps}$ of $\ob{E}_2$ such that $O_{\eps}\subset A_{\eps}$ and
\begin{align*}
 \Capa_{\ob{E_2}}(A_{\eps})& \le \Capa_{\ob{E_2}}(O_{\eps})+\eps.
 \end{align*}
 Therefore, $
 \Capa_{\ob{E_2}}(A_{\eps}) < (C(E_1)+1) \eps.
$ Since $O_{\eps} \subset A_{\eps}$, $f$ is continuous on $\ob{E}_1\setminus A_{\eps}$. This proves $f$ is $\Capa_{\ob{E_2}}$-quasi continuous on $\ob{E_1}$. The remaining claim can be proved in the same manner.
\qed\end{proof}

\subsection{Proof of Proposition~\ref{prop:1}}
\begin{lemma}\label{lem:ex}
For $a \in \partial D$, $B_{a}^{\ast}(1)$ is a bounded $W^{1,2}$-extension domain.
 In particular, there exists positive constant $C>0$ such that 
 \begin{equation*}
\text{\rm Cap}(A) \le C \cdot \text{\rm Cap}_{\ob{B_{a}^{\ast}(1)}}(A) \le C \cdot \text{\rm Cap}(A) , \quad A \subset \ob{B_{a}^{\ast}(1)}.
\end{equation*}
\end{lemma}
\begin{proof}
Since $B_{+}(1)$ is a bounded convex domain, this is a bounded Lipschitz domain. Therefore, $B_{+}(1)$ is a bounded $W^{1,2}$-extension domain. By Condition~\ref{as:1}, $B_{+}(1)$ and $B_{a}^{\ast}(1)$ are bi-Lipschitz homeomorphic. Therefore, it follows from \cite[Theorem~8]{HKT} that $B_{a}^{\ast}(1)$ is a bounded $W^{1,2}$-extension domain. Therefore, we see from Proposition~\ref{prop:comp} that  there exists a positive constant $C>0$ such that
\begin{equation}
\text{\rm Cap}(A) \le C \cdot \text{\rm Cap}_{\ob{B_{a}^{\ast}(1)}}(A), \quad A \subset \ob{B_{a}^{\ast}(1)} \label{eq:addd}.
\end{equation}
 On the otherhand, it holds that
\begin{equation}
\text{\rm Cap}_{\ob{B_{a}^{\ast}(1)}}(A) \le \text{\rm Cap}(A), \quad A \subset \ob{B_{a}^{\ast}(1)}  \label{eq:addd2}
\end{equation}
by Lemma~\ref{lem:comp}.  It is clear that \eqref{eq:addd} and  \eqref{eq:addd2} imply ``In particular'' part.
\qed\end{proof}
Since $\partial D \subset \bigcup_{a \in \partial D} \Psi_{a}(B(1))$, there exists $\{a_{n}\}_{n=1}^{\infty}$ of $\partial D$ such that $\partial D \subset \bigcup_{n=1}^{\infty}\Psi_{a_n}(B(1)) $. In the sequel, for each $n \in \N$, we shall write $B_{n}$ for $B_{a_{n}}^{\ast}(1)$.
\begin{proof}[Proof of Proposition~\ref{prop:1}]
Fix $n \in \N$. We note that $B_{+}(1)$ is a bounded Lipschitz domain. By \cite[Theorem~1~(i), p.~133]{EG}, there exists a positive constant $c$  such that
\begin{equation}
\mathcal{H}^{d-1}(\partial B_{+}(1)) \le c\|\bone_{B_{+}(1)}\|_{W^{1,2}(B_{+}(1))}^2=cm(B_{+}(1))^2<\infty. \label{eq:boundedtrace}
\end{equation}
By Condition~\ref{as:1}, $\partial B_{+}(1)$ and $ \partial B_n$ are bi-Lipschitz homeomorphic. Thus, we see from  \cite[Theorem~1, p.~75]{EG} and \eqref{eq:boundedtrace} that
$\mathcal{H}^{d-1}(\partial B_n)<\infty.$
Since $\partial D \subset \bigcup_{n=1}^{\infty}\partial B_n$, $\sg$ is a Radon measure on $(\partial D, \mathcal{B}(\partial D))$.

Take an $A \subset \partial D$ with $\Capa_{\ob{D}}(A)=0$. We see from Lemma~\ref{lem:comp} that
\begin{align*}
\Capa_{\ob{B_n}}(A \cap \ob{B_n})
& \le \Capa_{\ob{D}}(A \cap \ob{B_n}) \le \Capa_{\ob{D}}(A)= 0. 
\end{align*}
By using Lemma~\ref{lem:ex}, we have $\Capa (A \cap \ob{B_n})=0$.
From \cite[Theorem~4, p.~156]{EG}, we have $\sg(A \cap \ob{B_n}) = \mathcal{H}^{d-1}(A \cap \ob{B_n})=0$. Recall that it holds that $\ob{D}=\bigcup_{n=1}^{\infty}B_n$. This implies that
\begin{equation*} 
\sg(A) =\sg \left(A \cap \bigcup_{n=1}^{\infty}\ob{B_n} \right) \le \sum_{n=1}^{\infty} \sg(A \cap \ob{B_n})=0,
\end{equation*}
which comletes the proof. \qed
\end{proof}

\section{Proof of Theorem~\ref{thm:1}}

In this section, using the result of \cite{SG}, we prove Theorem~\ref{thm:1}. For the proof, we shall give some lemmas.
From Condition~\ref{as:2}, there are increasing bounded open subsets $\{U_n\}_{n=1}^{\infty}$ of $\R^d$ satisfy the following conditions: 
\begin{itemize}
\item $I_{n}:=U_{n} \cap D$ is a bounded Lipschitz domain of $\R^d$;
\item $\ob{D}=\bigcup_{n=1}^{\infty}O_n$, where we define $O_{n}:=U_n \cap \ob{D}$.
\end{itemize}
The closure of $I_n$ in $\R^d$ is denoted by $J_{n}$. Note  that $J_{n}$ is a compact subset of $\ob{D}$.

Since $I_n$ is a bounded Lipschitz domain of $\R^d$, there exists a Hunt process $Y^n=(\{Y_{t}^{n}\}_{t \ge 0},\{Q_{x}^{n}\}_{x \in J_{n}})$ on $J_n$ which has the following properties (cf \cite[Theorem~3.1, Theorem~3.4]{BH}): 
\begin{enumerate}
\item[(Y.1)]
the Dirichlet form of $Y^n$ is regular on $L^{2}(J_{n},m)$ and expressed as
\begin{align}
\mathcal{A}^{n}(f,g):=\frac{1}{2}\int_{I_n}(\nabla f, \nabla g)\,dx ,\quad f,g \in 
\mathcal{D}(\mathcal{A}^{n}):=H^{1}(I_n), \label{eq:recall}
\end{align}
\item[(Y.2)]
$Y^n$ has a transition density $q_{t}^{n}(x,y)$ which is continuous on $(0,\infty) \times J_{n} \times J_{n}$,
\item[(Y.3)]
for any $f \in \mathcal{B}_{b}(J_n)$, we have $q_{t}^{n}f:=\int_{J_n}q_{t}^{n}(\cdot,y)f(y)\,dy \in C_{b}(J_n)$,
\item[(Y.4)] there exist $a_{1,n}=a_{1,n}(d,I_n)>0, a_{2,n}=a_{2,n}(d,I_n)>0$, and $b_{1,n}=b_{1,n}(d,I_n)>0, b_{2,n}=b_{2,n}(d,I_n)>0$ such that 
\begin{align}
b_{1,n}t^{-d/2} \exp(-|x-y|^2&/b_{2,n}t)
\le q_{t}^{n}(x,y)  \label{eq:hkeest} \\
&\le a_{1,n}t^{-d/2} \exp(-|x-y|^2/a_{2,n}t)  \notag
\end{align}
for any $(t,x,y) \in (0,\infty) \times J_n \times J_n$.
\end{enumerate}
We denote by $(\mathcal{L}^n,\mathcal{D}(\mathcal{L}^n))$ the (non-positive) $L^2$-generator of $(\mathcal{A}^{n},\mathcal{D}(\mathcal{A}^{n}))$. The semigroup of $Y^n$ is canonically extended to semigroups on $L^{1}(\ob{D},m)$ and $L^{2}(\ob{D},m)$. The extensions are also denoted by $\{q_t^n\}_{t>0}$. The $L^1$-generator of $\{q_t^n\}_{t>0}$ is denote by $(\mathcal{L}^n_1,\mathcal{D}(\mathcal{L}^n_1))$.

Since $I_n$ is a bounded Lipschitz domain, it is also an extension domain in the sense of \cite{HKT}. Therefore, by \cite[Theorem~2]{HKT}, there exist positive constants $c \ge 1$, $R>0$ such that
$$c^{-1}r^{d} \le m(B(x,r) \cap J_n)) \le cr^d$$
for any $x \in J_n, r \in (0,R]$. This means $J_n$ is an Ahlfors $d$-space in the sense of \cite{BKK}. $Y^n$ is a diffusion process on $J_n$ with Gaussian bounds \eqref{eq:hkeest}. Thus, we can apply \cite[Proposition~A.3]{BKK} and obtain the next lemma.
\begin{lemma}\label{lem:branch}
Let $n \in \N$ and $K$ be a compact subset of $J_n$, and $\eps>0$. Then there exists a function $f_n \in  L^{\infty}(J_n,m)$ such that $f_n(x)=1$ for $x \in K$, $f_n(x)=0$ when $\text{dist}(x,K) \ge \eps$, and $f \in \mathcal{D}(\mathcal{L}^n)$. Furthermore, $0 \le f_n \le 1$ and $\|\mathcal{L}^nf_n \|_{L^{\infty}(J_n,m)}<\infty$. 
\end{lemma}

We fix  $n \in \N$ and take the function $f_n$ in Lemma ~\ref{lem:branch}, whose support is included in $I_n$. Since $f_n \in \mathcal{D}(\mathcal{L}^n)$, there exists $t>0$ and $g \in L^{2}(J_n,m)$ such that $f_n=q_t^ng=q_{t/2}^n(q_{t/2}^ng)$. For each $t>0$, $q_{t/2}^ng \in L^{\infty}(J_n,m)$ by \eqref{eq:hkeest}. Thus, $q_t^ng=q_{t/2}^n(q_{t/2}^ng)$ is a bounded continuous function on $J_n$ by (Y.3). This implies that there exists a continuous version of $f_n$ on $\ob{D}$. The continuous version is also denoted by $f_n$.

\begin{lemma}\label{lem:branch2}
For any $n \in \N$, $(\nabla f_n, \nabla f_n) \in L^{\infty}(\ob{D},m)$.
\end{lemma}
\begin{proof}
By the construction of $f_n$, we have $\|f_n\|_{L^{\infty}(J_n,m)}$ and $ \|\mathcal{L}^nf_n \|_{L^{\infty}(J_n,m)}<\infty$. 
It follows from \cite[Lemma~5.2~(ii)]{SG} that $f_n^2 \in \mathcal{D}(L^n_1)$. Thus, $f_n^2=q_{t}^ng$ for some $t>0$ and $g \in L^{1}(J_n,m)$. Since $q_{t/2}^n g \in \mathcal{D}(L^n_1)$, $$\mathcal{L}^n_1 f_n^2=\mathcal{L}^n_1(q_{t/2}^nq_{t/2}^n g)=q_{t/2}^n(\mathcal{L}^n_1q_{t/2}^n g).$$ 
We see from \eqref{eq:hkeest} that $q_{t/2}^n (L^{1}(J_n,m)) \subset L^{\infty}(J_n,m)$.
Since $\mathcal{L}^n_1q_{t/2}^n g \in L^{1}(J_n,m)$, it holds that $\mathcal{L}^n_1 f_n^2=q_{t/2}^n(\mathcal{L}^n_1q_{t/2}^n g) \in L^{\infty}(J_n,m)$. By \cite[Lemma~2.5~(ii)]{SG},
$$(\nabla f_n,\nabla f_n)=\mathcal{L}_1^nf_n^2-2f_n\mathcal{L}^nf_n,$$
which yields $(\nabla f_n,\nabla f_n) \in L^{\infty}(J_n,m)$. Since the support of $f_n$ is included in $I_n$, it holds that $(\nabla f_n,\nabla f_n) \in L^{\infty}(\ob{D},m)$.
\qed
\end{proof}
In the sequel, we denote by $(\mathcal{L},\mathcal{D}(\mathcal{L}))$ the (non-postive) $L^2$-generator of $(\cA,H^{1}(D))$.
\begin{lemma}\label{lem:branch3}
For any $n \in \N$, $f_n \in \mathcal{D}(\mathcal{L})$ and $\mathcal{L}f_n \in L^{\infty}(\ob{D},m)$. 
\end{lemma}
\begin{proof}
Since the support of $f_n$ is included in $I_n$, $f_n$ belongs to $H^{1}(D)$. For any $g \in H^{1}(D) \subset H^{1}(I_n)$,
\begin{align}
\cE(f_n,g)=(1/2)\int_{D}(\nabla f_n, \nabla g)\,dm =(1/2)\int_{I_n}(\nabla f_n, \nabla g)\,dm=\mathcal{A}^{n}(f_n,g).\label{localneumann}
\end{align}
Since $f_n \in \mathcal{D}(\mathcal{L}^n)$, $\cE(f,g)=-\int_{I_{n}}g\mathcal{L}^nf_n\,dm.$ This shows that the functional $H^{1}(D) \ni g \mapsto \cE(f,g)$ is continuous with respect to the $L^{2}(D,m)$-topology.Therefore, $f_n \in \mathcal{D}(\mathcal{L})$. It follows from \eqref{localneumann} that 
\begin{equation*}
\int_{I_n}g\mathcal{L} f_n\,dm=\int_{I_{n}}g \mathcal{L}^nf_n\,dm
\end{equation*}
for any $g \in H^{1}(D)$ such that $\tilde{g}^{\ob{D}}=0\ \Capa_{\ob{D}}\text{-q.e. on }\ob{D} \setminus O_n.$ The whole of such a function $g$ is a dense subspace of $L^{2}(O_n,m)$ and so of $L^{2}(I_n,m)$. See \cite[Theorem~4.4.3~(i)]{FOT} for the proof. This implies that $\mathcal{L} f_n=\mathcal{L}^nf_n$ $m$-a.e. on $I_n$. and $\mathcal{L} f_n$ is bounded on $J_n$ by Lemma~\ref{lem:branch}. By \cite[Lemma~5.2~(i)]{SG}, the support of $\mathcal{L}f_n$ is included in that of $f_n$. Therefore, $\mathcal{L}f_n$ is a bounded function on $J_n$. \qed
\end{proof}

Let $\ob{D}^{\mathbb{Q}}=\ob{D} \cap \mathbb{Q}^d$.
According to Lemma~\ref{lem:branch}, Lemma~\ref{lem:branch2}, and Lemma~\ref{lem:branch3}, we reach the next proposition.
\begin{proposition}\label{prop;branch}
There exist $\{f_n\}_{n=1}^{\infty} \subset \mathcal{D}(\mathcal{L}) \cap C_{c}(\ob{D})$ such that
\begin{itemize}
\item[(i)] for all $\eps \in (0,1) \cap \mathbb{Q}$ and $y \in \ob{D}^\mathbb{Q}$ there exists $n \in \N$ such that $f_n(x) \ge 1$, for all $x \in \ob{D} \cap \ob{B(y,\eps/4)}$ and $f_n=0$ on $\ob{D} \setminus B(y,\eps/2)$,
\item[(ii)] for any $n \in \N$, $\mathcal{L}f_n \in L^{\infty}(\ob{D},m)$ and $(\nabla f_n, \nabla f_n) \in L^{\infty}(\ob{D},m)$.
\end{itemize}
\end{proposition}

Theorem~\ref{thm:wrsf} implies that $G_{\alpha}^0\bone_{A}(x)=0$ for any $\alpha>0, x \in \ob{D}$, and any $A \in \mathcal{B}(\ob{D})$ with $m(A)=0$. This implies that the kernel of $G_{\alpha}^0$ is absolutely continuous with respect to $m$ and there exists a jointly measurable function $r_{\alpha}(x,y)$ on $\ob{D} \times \ob{D}$ such that
\begin{equation*}
G_{\alpha}^0f(x)=\int_{\ob{D}}r_{\alpha}^0(x,y)f(y)\,m(dy)
\end{equation*} 
for any $\alpha>0$ and $x \in \ob{D}$ and any $f \in L^{1}(\ob{D},m) \cap L^{\infty}(\ob{D},m)$. We note that $(\cE,H^{1}(D))$ is conservarive. See \cite[Exercise~5.7.1]{FOT} for the conservativeness. Thus, it holds that $\alpha G_{\alpha}^{0}1(x)=1$ for any $x \in \ob{D}$ and $\alpha>0$. 
For each $\alpha>0$, the density $r_{\alpha}^0(x,y)$ is symmetric in $x$ and $y$, and satisfies the resolvent equation. Let $\{p_{t}^0(x,y)\}_{t>0}$ be the jointly measurable functions on $\ob{D} \times \ob{D}$ whose Laplace transform is $\{r_{\alpha}^0(x,y)\}_{\alpha>0}$. Then, it holds that $\int_{\ob{D}}p_{t}^0(x,y)\,m(dy)=1$ for any $t>0$ and $x \in \ob{D}$. It is easy to see that each $p_t^0(x,y)$ is symmetric in $x$ and $y$, and satisfies the Chapman-Kolmogorov equation. By the Kolmogorov extension theorem, we can construct a family of probability measure $\{P_x\}_{x \in \ob{D}}$ on $\ob{D}^{[0,\infty)}$ and a Markov process $X^0=(\{X_t^0\}_{t \ge 0}, \{P_x\}_{x \in \ob{D}})$ on $\ob{D}$ with respect to $\{p^0_t(x,y)\}_{t>0}$. By the construction of $X^0$, the resolvent $\{R_{\alpha}^0\}_{\alpha>0}$ generates the Dirichlet form $(\cE,H^{1}(D))$ and $R_{\alpha}^0f \in C_{b}(\ob{D})$ for any $\alpha>0, f \in L^{1}(\ob{D},m) \cap L^{\infty}(\ob{D},m)$. On the other hand, since $(\cE,H^{1}(D))$ is a regular conservative strong local Dirichlet form on $L^{2}(\ob{D},m)$, there exists a conservative diffusion process $X^1=(\{X_t^1\}_{t \ge 0}, \{P_x^1\}_{x \in \ob{D}})$ on $\ob{D}$ whose resolvent $\{R_{\alpha}^1\}_{\alpha>0}$ satisfies
\begin{equation*}
R_{\alpha}^{1}f(x)=G_{\alpha}^{0}f(x),\quad x \in \ob{D}\setminus N,\ \alpha>0,\ f \in L^{1}(\ob{D},m) \cap L^{\infty}(\ob{D},m).
\end{equation*}
Here, $N$ is a subset of $\ob{D}$ with $\Capa_{\ob{D}}(N)=0$. Since $\{R_\alpha^0\}_{\alpha>0}$ is a version of $\{G_{\alpha}^0\}_{\alpha>0}$, there exists $N_1 \in \mathcal{B}(\ob{D})$ with $m(N_1)=0$ such that the distributions of $X^0=(\{X_t^0\}_{t \ge 0}, \{P_x\}_{x \in \ob{D}})$ and $X^1=(\{X_t^1\}_{t \ge 0}, \{P_x^1\}_{x \in \ob{D}})$ coincide except on $N_1$. 
Since $X^1$ is a conservative diffusion process on $\ob{D}$,
\begin{equation}
P_x(C([0,\infty);\ob{D}))=1,\quad x \in \ob{D} \setminus N_1. \label{eq:diffusion}
\end{equation}
Here, $C([0,\infty);\ob{D})$ denotes the space of $\ob{D}$-valued continuous functions on $[0,\infty)$. For a subset $S \subset [0,\infty)$ with $\inf S=\eps>0$ and $\sup S<\infty$, we define $B_S=\{X^0 \in \ob{D}^{[0,\infty)} \mid X \text{ is continuous on }S \}$. Then, by \eqref{eq:diffusion} and the Markov property of $X^0$, it holds that
\begin{align*}
P_{x}(B_S)=E_{x}[P_{X_\eps^0}(B_{S-\eps})]=\int_{\ob{D}}P_{y}(B_{S-\eps})p_{\eps}^0(x,y)\,m(dy)=\int_{\ob{D}}p_{\eps}^0(x,y)\,m(dy)=1
\end{align*}
for any $x \in \ob{D}$. The same argument in \cite[Lemma~2.1.2]{SV} shows that
\begin{equation}
P_x(C((0,\infty);\ob{D}))=1,\quad x \in \ob{D}. \label{eq:notbranch}
\end{equation}
Here, $C((0,\infty);\ob{D})$ denotes the space of $\ob{D}$-valued continuous functions on $(0,\infty)$. 

We are now ready to prove Theorem~\ref{thm:1}.
\begin{proof}[Proof of Theorem~\ref{thm:1}]
We denote by $\{p_t^0\}_{t>0}$ the semigroup of $X^0$. Recall that each $p_{t}^0$ has a jointly measurable density $p_{t}^0(x,y)$ defined on $\ob{D} \times \ob{D}$. Therefore, the condition (H.1) stated in \cite{SG} is satisfied for $\{p_t^0\}_{t>0}$. For Theorem~\ref{thm:1}, it is sufficient to prove that the conditions (H.2)'~(i), (ii), (iii), (iv) stated in \cite{SG}. (H.2)' (iv) is clear from  the construction of $\{p_t^0\}_{t>0}$ and \eqref{eq:notbranch}. Since $R^0_{\alpha}(L^{1}(\ob{D},m) \cap L^{\infty}(\ob{D},m)) \subset C_{b}(\ob{D})$, the condition (H.2)' (iii) is satisfied. By Proposition~\ref{prop;branch}, the conditions (H.2)'~(i) and (ii) are satisfied. See also \cite[Remark~2.7~(ii)]{SG}. By \cite[Lemma~2.9]{SG}, there exists a Hunt process whose semigroup is $\{p_{t}^0\}_{t>0}$.
\qed
\end{proof}

\section{Proof of Theorem~\ref{thm:doubf}}
In the following, we denote $\ob{D}_{\Delta}=\ob{D} \cup \{\Delta \}$ by the one-point compactification of $\ob{D}$. Any funtion $f$ defined on $\ob{D}$ is extended to $\ob{D}_{\Delta}$ by setting $f(\Delta)=0$.
\subsection{An estimate of boundary local time}
Let $X^0=(\Omega,\{X^0_{t}\}_{t \ge 0},\{ P_x\}_{x \in \ob{D}})$ be the Hunt process  in Theorem~\ref{thm:1}. In the sequel, for $x \in \ob{D}$, we denote by $E_{x}$ the expectation under the measure $P_x$. The semigroup of $X^0$ is denoted by $\{p_{t}^0\}_{t > 0}$.
We note that $X^{0}$ satisfies the  absolutely continuous condition: the transition  function  $p_{t}^0(x,\cdot)$ of $X^0$ satisfies that 
\begin{align*}
&p_{t}^0(x,\cdot)\text{ is absolutely continuous with respect to }m\text{ for each $t>0$ and $x \in \ob{D}$.}
\end{align*}
Recall that $\beta$ is a locally bounded nonnegative Borel measurable function on $\partial D$ and $\{L_t\}_{t \ge0}$ a positive continuous additive functional with Revuz measure $\sg$. 
 
This section is devoted to prove the following proposition.
\begin{proposition}\label{blt1}
For any compact subset $K$ of $\ob{D}$, we have
\begin{align*}
\lim_{t \to 0}\sup_{x \in K}E_{x}\left[1-\exp \left(-\int_{0}^{t}\beta (X_s^0)\,dL_s \right) \right]=0.
\end{align*}
\end{proposition}

To prove Proposition~\ref{blt1}, we give some lemmas. Recall that $\{U_n\}_{n=1}^{\infty}$  are increasing bounded open subsets of $\R^d$ with the following conditions:
\begin{itemize}
\item $I_{n}:=U_{n} \cap D$ is a bounded Lipschitz domain of $\R^d$;
\item $\ob{D}=\bigcup_{n=1}^{\infty}O_n$, where we define $O_{n}:=U_n \cap \ob{D}$.
\end{itemize}
The closure of $I_n$ in $\R^d$ is denoted by $J_{n}$. Note that $O_{n}$ is an open subset of $J_{n+1}$ and $\ob{D}$. 
For each $n \in \mathbb{N}$, we define
\begin{equation*}
X_{t}^{n}:=
\begin{cases}
X_t^0 &\text{ if } t<\tau_n, \\
\Delta &\text{ if } t \ge \tau_n,
\end{cases}
\end{equation*}
where $\tau_{n}=\inf \{t>0:X_{t}^{0} \in \ob{D} \setminus O_n\}$. We call $X^n=(\Omega,\{X_{t}^n\}_{t \ge0}, \{P_x\}_{x \in O_n})$ the part of $X^{0}$ on $O_n$. We note that 
\begin{align*}
p_{t}^{n}f(x)=E_{x}[f(X_t^0):t<\tau_n],\quad f \in \mathcal{B}_{b}(\ob{D}),\, x \in O_n
\end{align*}  
is the semigroup of $X^n$. It is clear that $X^n$ satisfies the absolutely continuous condition.
The Dirichlet form of $X^n$ is a regular Dirichlet form on $L^{2}(O_n,m)$ and it is expressed as
\begin{align*}
\mathcal{E}^{n}(f,g)&=\frac{1}{2}\int_{D}(\nabla f, \nabla g)\,dx ,\quad
\mathcal{D}(\mathcal{E}^{n})=\left\{ 
f \in H^{1}(D) \relmiddle| \tilde{f}^{\ob{D}}=0,\, \Capa_{\ob{D}}\text{-q.e. on }\ob{D} \setminus O_n\right\}.
\end{align*}
See \cite[Theorem~4.4.2]{FOT} and \cite[Theorem~4.4.3]{FOT} for details. 

Recall that $Y^n=(\{Y_t^n\}_{t \ge 0},\{Q_x^n\}_{x \in J_n})$ is the Hunt process which satisfies (Y.1), (Y.2), (Y.3), (Y.4) stated in the previous section. In the sequel, we denote by $$Y^{n+1,n}=(\{Y_{t}^{n+1,n}\}_{t \ge 0}, \{Q_{x}^{n+1}\}_{x \in K_{n}})$$ the part  of $Y^{n+1}$ on $O_n$. 
 $Y^{n+1,n}$ is defined in the same manner as $X^n$.
The semigroup of $Y^{n+1,n}$ is denoted by $\{q_{t}^{n+1,n}\}_{t>0}$.

In fact, the finite dimensional distributions of $X^n$ and $Y^{n+1,n}$ coincide for any starting point. To show this, we prepare some lemmas.

\begin{lemma}\label{lem:identification}
For any $n \in \mathbb{N}$, the Dirichlet form of $Y^{n+1,n}$ coincides with that of $X^n$.
\end{lemma}
\begin{proof}
By \cite[Theorem~4.4.2]{FOT} and \cite[Theorem~4.4.3]{FOT}, the Dirichlet form of $Y^{n+1,n}$ is regular on $L^{2}(O_n,m)$ and expressed as 
\begin{align*}
\mathcal{A}^{n+1,n}(f,g)&=\frac{1}{2}\int_{I_{n+1}}(\nabla f, \nabla g)\,dx,\\ \quad
\mathcal{D}(\mathcal{A}^{n+1,n})&=\left\{ f \in H^{1}(I_{n+1})  \relmiddle|
\tilde{f}^{J_{n+1}}=0,\, \Capa_{J_{n+1}}\text{-q.e. on }J_{n+1} \setminus O_n \right\}.
\end{align*}
First, we prove $\mathcal{D}(\mathcal{A}^{n+1,n}) \subset \mathcal{D}(\cA^n)$. Take an $f  \in \mathcal{D}(\mathcal{A}^{n+1,n})$. Then, there exists a $\Capa_{J_{n+1}}$-quasi continuous version $\tilde{f}^{J_{n+1}}$ such that $\tilde{f}^{J_{n+1}}=0$, $\Capa_{J_{n+1}}$-q.e. on $J_{n+1}\setminus O_n$. From Proposition~\ref{prop:comp} and \ref{prop:comp2}, $\tilde{f}^{J_{n+1}}$ is $\Capa_{\ob{D}}$-quasi continuous on $J_{n+1}$ and $\tilde{f}^{J_{n+1}}=0$, $\Capa_{\ob{D}}$-q.e. on $J_{n+1}\setminus O_n$  Define $g :\ob{D} \to \R$ by 
\begin{equation*}
\begin{cases}
g=\tilde{f}^{J_{n+1}} \text{ on } J_{n+1}, \\
g=0 \text{ on } \ob{D} \setminus  J_{n+1}. 
\end{cases}
\end{equation*}
Then, $g$ is a $\Capa_{\ob{D}}$-quasi continuous on $\ob{D}$ and $g=0$, $\Capa_{\ob{D}}$-q.e. on $\ob{D} \setminus O_n$. Define $h \in H^{1}(D)$ by 
\begin{equation*}
\begin{cases}
h=f \text{ on } I_{n+1}, \\
h=0 \text{ on } D \setminus  I_{n+1}. 
\end{cases}
\end{equation*}
Then, $g$ is a $\Capa_{\ob{D}}$-quasi continuous version of $h$. Therefore, $h \in D(\cA^n)$. Since $f=h$, $m$-a.e. on $O_n$, we have $\mathcal{D}(\mathcal{A}^{n+1,n}) \subset \mathcal{D}(\cA^n)$. 

Next we prove  $\mathcal{D}(\cA^n) \subset \mathcal{D}(\mathcal{A}^{n+1,n})$. Take an $f \in \mathcal{D}(\cA^n)$. Then there exists a $\Capa_{\ob{D}}$-quasi continuous version $\tilde{f}^{\ob{D}}$ such that $\tilde{f}^{\ob{D}}=0$, $\Capa_{\ob{D}}$-q.e. on $\ob{D} \setminus O_n$. Define $g :J_{n+1} \to \R$ by 
$
g=\tilde{f}^{\ob{D}}|_{J_{n+1}}.
$
From Proposition~\ref{prop:comp} and \ref{prop:comp2}, $g$ is $\Capa_{J_{n+1}}$-quasi continuous and $g=0$, $\Capa_{J_{n+1}}$-q.e. on $J_{n+1}\setminus O_n$. Define $h \in H^{1}(I_{n+1})$ by $h=f |_{I_{n+1}}.$ Since $g$ is a $\Capa_{J_{n+1}}$-quasi continuous version of $h$, we have $h \in \mathcal{D}(\mathcal{A}^{n+1,n})$. Since $f=h,$ $m$-a.e. on $O_n$, we have the claim.
\qed\end{proof}

\begin{lemma}\label{lem:partstrf}
For any $n \in \mathbb{N}$, the part process $Y^{n+1,n}$ has a semigroup strong Feller property. That is, for any $f \in \mathcal{B}_{b}(O_n)$ and $t>0$, we have $q_{t}^{n+1,n}f \in C_{b}(O_n)$.
\end{lemma}
\begin{proof}
Since $Y^{n+1}$ has property (Y.3) and $J_{n+1}$ is compact, the proof is complete by \cite[Theorem~1.4]{CK}.
\qed\end{proof}
We shall show the finite dimensional distributions of $X^n$ and $Y^{n+1,n}$ coincide for any starting point.
\begin{lemma}\label{lem:partstrf2}
For any $n \in \mathbb{N}$, $f \in \mathcal{B}_{b}(O_n)$,  and $t>0$,
\begin{align*}
&p_{t}^{n}f(x)=q_{t}^{n+1,n}f(x),\quad x \in O_n. 
\end{align*}
In particular, the part process $X^n$ has the semigroup strong Feller property.
\end{lemma}
\begin{proof}
From~Lemma~\ref{lem:identification}, for any $f \in C_{b}(O_n)$ and $t>0$, we have
$
p_{t}^{n}f=q_{t}^{n+1,n}f$ $m$-a.e. on $O_n 
$.
It follows from the absolute continuity condition of $X^n$ that
\begin{align*}
p_{s+t}^{n}f(x)&=p_{s}^{n}(p_{t}^{n}f)(x)=p_{s}^n(q_{t}^{n+1,n}f)(x)
\end{align*}
for all $s>0$, $x \in O_n$. From~Lemma~\ref{lem:partstrf}, $q_{t}^{n+1,n}f \in C_{b}(O_n)$. Letting $s \to 0$, we have 
$
p_{t}^{n}f(x)=q_{t}^{n+1,n}f(x)
$
from the sample path continuity of $X^n$. Using a monotone class theorem, we obtain the claim. ``In particular'' part follows from Lemma~\ref{lem:partstrf}.
\qed\end{proof}

By using Lemma~\ref{lem:partstrf}, we give some estimates necessary for the proof of Proposition~\ref{blt1}.

\begin{lemma}\label{blt2}
For any $n \in \mathbb{N}$, there exists $a_{3,n}=a_{3,n}(d,I_{n})>0$ such that
\begin{align*}
\sup_{x \in O_n}\int_{0}^{t}\int_{\partial D \cap  O_n}q_{s}^{n+1}(x,y)\,\sg(dy)\,ds \le a_{3,n} \sqrt{t},\quad 0<t\le 1.
\end{align*}
\end{lemma}
\begin{proof}
Since $\partial D  \cap O_n$ and $J_n$ are bounded subsets of $\ob{D}$, from Condition~\ref{as:2}, there exists a bounded open subset $U_n \subset \R^d$ such that $J_n \cup (\partial D  \cap O_n) \subset U_n$ and $D \cap U_n$ is a bounded Lipschitz domain of $\R^d$. In the sequel, $D \cap U_n$ is denoted by $V_n$. It is easy to see $\partial D  \cap O_n \subset \partial V_n$ and $O_n \subset  \ob{V_n}$. Define the measure $\sg_{n}$ on $\partial V_n$ by $\bone_{\partial V_n} \cdot \mathcal{H}^{d-1}$. Then, 
\begin{align}
 \int_{0}^{t}\int_{\partial D \cap  O_n}q_{s}^{n+1}(x,y)\,\sg(dy)\,ds
 &\le  \int_{0}^{t}\int_{\partial V_{n}}q_{s}^{n+1}(x,y)\,\sg_{n}(dy)\,ds. \label{eq:revision}
 \end{align}
 For $\eps>0$, we define $V_{n}^{\eps}=\{x \in V_{n} \mid \text{dist}(x, \partial V_{n})<\eps \}$. By \eqref{eq:hkeest} and Lemma~\ref{lem:EPH} below, there exist positive constants $\eps_{0}$ and $a_{3,n}=a_{3,n}(d,V_n)$ such that
 \begin{align}
 \frac{1}{\eps}\int_{V_n^{\eps}}q_{s}^{n+1}(x,y)\,dy \le a_{3,n}/\sqrt{s},\quad s \in (0,1].\label{eq:eqHPH}
  \end{align}
for any $\eps \in (0,\eps_0),\ x \in \ob{V_n}$, and $s \in (0,1]$. For each $s \in (0,1]$ and $x \in J_{n+1}$, $q_{s}^{n+1}(x,y)$ is a bounded continuous function in $y$. Thus, by letting $\eps \to 0$ in \eqref{eq:eqHPH}, we obtain
 \begin{align}
\sup_{x \in \ob{V_n}}\int_{\partial V_n}q_{s}^{n+1}(x,y)\,\sg_n(dy) \le a_{3,n}/\sqrt{s},\quad s \in (0,1] \label{eq:revision2}
  \end{align}
from \cite[Lemma~7.1]{CF}. Combining \eqref{eq:revision} with \eqref{eq:revision2}, we complete the proof.
\qed\end{proof}

\begin{lemma}\label{partdeal}
For any $0 \le s<t$, $n \in \N$, $f \in \mathcal{B}_{+}(\ob{D})$, $x \in O_n$,
\begin{align}
&E_{x}\left[\int_{s}^{t}f(X_r^0)\,dL_{r \wg \tau_n}\right]=E_{x}\left[\int_{s}^{t}f(X_{r}^{n})\,dL_{r} \right], \label{eq:partdeal}\\
&E_{x}\left[\int_{s}^{t}f(X_r^0)\,dL_{r \wg \tau_n}\right]\le E_{x}\left[\int_{s}^{t}f(X_{r-s}^{n}\circ \theta_s)\,dL_{r-s}\circ \theta_s: s<\tau_n \right].\label{eq:partdeal1}
\end{align}
\end{lemma}
\begin{proof}
By straightforward calculation,
\begin{align}\label{eq:eqdeal2}
&E_{x}\left[\int_{s}^{t}f(X_r^0)\,dL_{r \wg \tau_n}\right]\\
& =E_{x}\left[\int_{s}^{t}f(X_r^0)\,dL_{r \wg \tau_n}: t < \tau_n \right]+E_{x}\left[\int_{s}^{t}f(X_r^0)\,dL_{r \wg \tau_n}: t \ge \tau_n \right] \notag \\
&=E_{x}\left[\int_{s}^{t}f(X_r^n)\,dL_{r}:  t < \tau_n \right]+E_{x}\left[\int_{s}^{\tau_n}f(X_r^n)\,dL_{r}:  t \ge \tau_n \right] \notag \\
&\quad+E_{x}\left[\int_{\tau_n}^{t}f(X_r^0)\,dL_{r \wg \tau_n}:  t \ge \tau_n \right] \notag \\
&=E_{x}\left[\int_{s}^{t}f(X_r^n)\,dL_{r}:  t < \tau_n \right]+E_{x}\left[\int_{s}^{\tau_n}f(X_r^n)\,dL_{r}:  t \ge \tau_n \right]+0.\notag
\end{align}
In the last line, we used the fact that the measure $dL_{r \wg \tau_n}$ vanishes on $[\tau_n,t]$. It holds that \begin{equation}E_{x}\left[\int_{\tau_n}^{t}f(X_r^n)\,dL_{r} :t \ge \tau_n \right]=E_{x}\left[\int_{\tau_n}^{t}f(\Delta)\,dL_{r} :t \ge \tau_n \right]=0. \label{eq:infinitypoint}
\end{equation}
Combining  \eqref{eq:eqdeal2} with \eqref{eq:infinitypoint}, we obtain
\begin{align*}
E_{x}\left[\int_{s}^{t}f(X_r^0)\,dL_{r \wg \tau_n}\right]&=E_{x}\left[\int_{s}^{t}f(X_r^n)\,dL_{r}:  t < \tau_n \right]+E_{x}\left[\int_{s}^{\tau_n}f(X_r^n)\,dL_{r}:  t \ge \tau_n \right]\\
&\quad+E_{x}\left[\int_{\tau_n}^{t}f(X_r^n)\,dL_{r} :t \ge \tau_n \right]\\
&=E_{x}\left[\int_{s}^{t}f(X_r^n)\,dL_{r} \right],
\end{align*}
which implies \eqref{eq:partdeal}. By a similar argument, 
\begin{align*}
E_{x}\left[\int_{s}^{t}f(X_r^0)\,dL_{r \wg \tau_n}\right]&=E_{x}\left[\int_{s}^{t}f(X_r^0)\,dL_{r \wg \tau_n}: s<\tau_n \right] \notag \\
&=E_{x}\left[\int_{s}^{t}f(X_{r-s}^n\circ \theta_s^n)\,dL_{r-s}\circ \theta_s: s<\tau_n, t < \tau_n \right] \notag \\
&\qad+E_{x}\left[\int_{s}^{\tau_n}f(X_{r-s}^n \circ \theta_s^n)\,dL_{r-s} \circ \theta_s:  s<\tau_n, t\ge  \tau_n \right].
\end{align*}
Here, $\theta_s^n$ is the shift operator of $X^n$. Since
$\theta_s^n=\theta_s$ for $ s<\tau_n$,  
\begin{align*}
E_{x}\left[\int_{s}^{t}f(X_r^0)\,dL_{r \wg \tau_n}\right]
&\le E_{x}\left[\int_{s}^{t}f(X_{r-s}^n\circ \theta_s)\,dL_{r-s}\circ \theta_s: s<\tau_n, t < \tau_n \right] \notag \\
&\qad+E_{x}\left[\int_{s}^{t}f(X_{r-s}^n \circ \theta_s)\,dL_{r-s} \circ \theta_s:  s<\tau_n, t\ge  \tau_n \right].
\end{align*}
This yields \eqref{eq:partdeal1}.
\qed\end{proof}

\begin{lemma}\label{blt3}
For any $0<t \le 1$, $n \in \mathbb{N}$, $f \in \mathcal{B}_{+}(\ob{D})$, $x \in O_n$, 
\begin{align*}
E_{x}\left[\int_{0}^{t}f(X_s^0)\,dL_{s \wedge \tau_{n}} \right]\le \int_{0}^{t}\int_{\partial D \cap O_n}q_{s}^{n+1}(x,y) f(y)\,\sg(dy)\,ds.
\end{align*}
\end{lemma}
\begin{proof}
From \cite[Lemma~5.1.10]{FOT}, we have, for any $f ,h \in \mathcal{B}_{+}(\ob{D})$,
\begin{align}
\int_{\ob{D}}h(x)E_{x}\left[\int_{0}^{t}f(X_s^0)\,dL_{s \wedge \tau_{n}} \right]\,dx=\int_{0}^{t}\int_{\partial D \cap K_{n}}E_{y}[h(X_s^{n})]f(y)\, \sg(dy)\,ds.\label{eq:locest}
\end{align}
It follows from Lemma~\ref{lem:partstrf2} that
\begin{align}
E_{y}[h(X_{s}^n)]= q_{s}^{n+1,n}h(y) \le \int_{J_{n+1}}h(x)q_{s}^{n+1}(y,x)\,dx.\label{eq:locest2}
\end{align}
for any $y \in O_n$, $s<t$. From \eqref{eq:locest}, \eqref{eq:locest2}, for any $f \in \mathcal{B}_{+}(\ob{D})$,
\begin{align}
E_{x}\left[\int_{0}^{t}f(X_s^0)dL_{s \wedge \tau_{n}} \right]\le \int_{0}^{t}\int_{\partial D \cap O_n}q_{s}^{n+1}(x,y) f(y)\,\sg(dy)\,ds\quad m\text{-a.e. }x \in J_{n+1}. \label{eq:blt3-2}
\end{align}
Fix $x \in O_n$ and $0<s<t$. Put $F(x)=E_{x}[\int_{0}^{t-s}f(X_r^0)\,dL_{r \wg  \tau_n} ]$.
 From \eqref{eq:partdeal1} and the Markov property of $X^0$, and Lemma~\ref{lem:partstrf2},
\begin{align}\label{eq:partlocal}
E_{x}\left[\int_{s}^{t}f(X_r^0)\,dL_{r \wedge \tau_{n}} \right] &\le E_{x}\left[\int_{s}^{t}f(X_{r-s}^{n}\circ \theta_s)\,dL_{r-s}\circ \theta_s: s<\tau_n \right] \\
&=E_{x}\left[ E_{X_{s}^0}\left[\int_{0}^{t-s}f(X_r^n)\,dL_{r} \right] : s<\tau_n \right] \notag \\
&=E_{x}\left[ E_{X_{s}^0}\left[\int_{0}^{t-s}f(X_r^0)\,dL_{r \wg \tau_n} \right] : s<\tau_n \right] \notag \\
&=p_{s}^{n}F(x)= q_{s}^{n+1,n}F(x) . \notag
\end{align}
In the third line, we used \eqref{eq:partdeal}. From the inequality $q_{s}^{n+1,n}F(x) \le q_{s}^{n+1}F(x)$ and \eqref{eq:blt3-2}, the right hand side of \eqref{eq:partlocal} is estimated as follows.
\begin{align}
q_{s}^{n+1,n}F(x)  
&\le \int_{J_{n+1}}q_{s}^{n+1}(x,y)\int_{0}^{t-s}\int_{\partial D \cap O_n}q_{r}^{n+1}(y,z) f(z)\,\sg(dz)\,dr\,dy. \label{eq:partlocal2}
\end{align}
Combining \eqref{eq:partlocal} with \eqref{eq:partlocal2}, we have
\begin{align*}
E_{x}\left[\int_{s}^{t}f(X_r^0)\,dL_{r \wedge \tau_{n}} \right] 
&\le \int_{J_{n+1}}q_{s}^{n+1}(x,y)\int_{0}^{t-s}\int_{\partial D \cap O_n}q_{r}^{n+1}(y,z) f(z)\,\sg(dz)\,dr\,dy \\
&=\int_{s}^{t}\int_{\partial D \cap O_n}q_{r}^{n+1}(x,z) f(z)\,\sg(dz)\,dr.
\end{align*}
Letting $s \to 0$, we complete the proof. 
\qed\end{proof}

\begin{lemma} \label{exit0}
For any $x \in \ob{D}$, $P_{x}( \lim_{n \to \infty} \tau_n <\infty)=0$.
\end{lemma}
\begin{proof}
We follow the argument in \cite[Lemma~5.10]{SG}.
We note that each $O_n$ is a relatively compact open subset of $\ob{D}$ and $\ob{D}=\bigcup_{n=1}^{\infty}O_n$.
By \cite[Lemma~5.5.2]{FOT}, there exists $N \subset \ob{D}$ such that $\Capa_{\ob{D}}(N)=0$ and
\begin{align}
P_{x}(\lim_{n \to \infty} \tau_n=\infty)=1,\quad x \in \ob{D} \setminus N. \label{eq:cons}
\end{align}
Recall that $X^0$ satisfies the absolutely continuous condition. By \cite[Theorem~4.1.1]{FOT} and \cite[Theorem~4.1.3]{FOT}, there exists a Borel subset $N_1 \subset \ob{D}$ such that $N \subset N_1$ and 
\begin{align}
P_{x}(\sg_{N_1}<\infty)=0, \label{eq:polar}
\end{align}
for all $x \in \ob{D}$. Here $\sg_{N_1}=\inf \{ t>0 \mid X_{t}^{0} \in N_1\}$. Take $x \in \ob{D}=\bigcup_{n=1}^{\infty}O_n$. Then, there exists $n_0 \in \N$ such that $x \in O_{n_0}$ and 
$
P_{x}(\Omega_1)=1,
$ where $\Omega_1:=\{\tau_{n_0}>0 \}$. For all $\omega \in \Omega_1$, $n>n_{0}$, and small $t=t(\omega)>0$, we have $\tau_{n} \circ \theta_{t}(\omega) \le \tau_n(\omega)$. Here, $\theta_{t}$ is the shift operator of $X^0$. Therefore, for all $\omega \in \Omega_1$, 
$
\lim_{t \to0}\lim_{n \to \infty}\tau_{n} \circ \theta_{t}(\omega) \le  \lim_{n \to \infty}\tau_n(\omega).
$
It follows from the Markov property of $X^0$, \eqref{eq:polar} and \eqref{eq:cons} that for any $x \in \ob{D}$
\begin{align*}
P_{x}(\lim_{n \to \infty}\tau_{n}<\infty) &\le P_{x}(\lim_{t \to0}\lim_{n \to \infty}\tau_{n} \circ \theta_{t}<\infty) \\
&\le \varliminf_{t \to0} E_{x}(P_{X_{t}^0}(\lim_{n \to \infty}\tau_{n}<\infty)) \\
&=\varliminf_{t \to0} E_{x}(P_{X_{t}^0}(\lim_{n \to \infty}\tau_{n}<\infty): X_{t}^0 \in \ob{D}\setminus N) \\
&=0.
\end{align*}
\qed\end{proof}

\begin{lemma} \label{exit1}
For any compact subset $K$ of $\ob{D}$ and $t>0$, we have
\begin{align*}
\lim_{n \to \infty }\sup_{x \in K}P_{x}(\tau_n \le t)=0.
\end{align*}
\end{lemma}
\begin{proof}
We follow the argument in \cite[Theorem~1.4]{CK}.
We may assume $K \subset O_1$. It follows from Lemma~\ref{lem:partstrf2} that for any $n \in \N$
\begin{align*}
P_{(\cdot)}(\tau_n \le t)=1-p_{t}^{n}\bone_{O_n}(\cdot) \in C_{b}(O_n).
\end{align*}
Therefore, the map $x \mapsto P_{x}(\tau_n \le t)$ is continuous on $K$. Hence, there exists $x_n \in K$ such that
$\sup_{x \in K}P_{x}(\tau_n \le t)=P_{x_n}(\tau_n \le t)$. Since $K$ is a compact subset of $\ob{D}$, there exists a subsequence of $\{ x_n\}_{n=1}^{\infty} \subset K$ which converges to some point $x_0 \in K$. This subsequence is also denoted by $\{x_n\}_{n=1}^{\infty}$. For any $n>m$, we have
$
P_{x_{n}}(\tau_n \le t) \le P_{x_n}(\tau_{m} \le t).
$
Since the map $x \mapsto P_{x}(\tau_m \le t)$ is continuous on $K$, 
\begin{align*}
\varlimsup_{n \to \infty}\sup_{x \in K}P_{x}(\tau_n \le t) = \varlimsup_{n \to \infty}P_{x_{n}}(\tau_n \le t) 
\le \varlimsup_{n \to \infty} P_{x_n}(\tau_{m} \le t) =P_{x_0}(\tau_{m} \le t).
\end{align*}
Letting $m \to \infty$, we complete the proof from Lemma~\ref{exit0}.
\qed\end{proof}

\begin{proof}[Proof of Proposition~\ref{blt1}]
We may assume $K \subset O_1$ and $ 0<t \le 1$. For any $n \in \mathbb{N}$,
\begin{align} \label{eq:strongfeller}
&\sup_{x \in K}E_{x}\left[1-\exp \left(-\int_{0}^{t}\beta(X_s^{0})\,dL_s \right) \right] \\
&\le \sup_{x \in O_n}E_{x}\left[1-\exp \left(-\int_{0}^{t \wg \tau_n}\beta(X_s^{0})\,dL_s\right) \right] +\sup_{x \in K}P_{x}(  t \ge \tau_n) \notag  \\
&\le  \sup_{x \in O_n }E_{x}\left[\int_{0}^{t \wg \tau_n}\beta(X_s^{0})\,dL_s \right] +\sup_{x \in K}P_{x}( 1 \ge \tau_n) \notag.
\end{align}
In the last inequality we used an elementary inequality: $1-\exp(-x) \le x$. Since $\beta$ is locally bounded, there exists $a_{4,n}>0$ such that $\sup_{z \in O_n} \beta(z) \le a_{4,n}$. Using Lemma~\ref{blt2} and Lemma~\ref{blt3}, we obtain
\begin{align*}
\sup_{x \in O_n }E_{x}\left[\int_{0}^{t \wg \tau_n}\beta(X_s^{0})\,dL_s \right] \le a_{3,n} a_{4,n} \sqrt{t}.
\end{align*}
Thus, letting $t \to 0$ and then $n \to \infty$ in \eqref{eq:strongfeller}, we complete the proof from Lemma~\ref{exit1}.
\qed\end{proof}

\subsection{Proof of Theorem~\ref{thm:doubf}~(i)}
Let $Y=( \{Y_t\}_{t \ge 0}, \{P_x^\beta\}_{x \in \ob{D}})$ be the subprocess of $X^0$ defined by the multiplicative functional $\{ \exp(-\int_{0}^{t} \beta(X_{s}^{0})\,dL_s)\}_{t \ge 0}$. We denote by $\{p_t\}_{t>0}$ the semigroup of $Y$. To prove Theorem~\ref{thm:doubf}~(i), we shall improve Theorem~\ref{thm:1} as follows.

\begin{theorem}\label{cor:corstf}
Suppose Condition~\ref{as:2}. Then, 
for any $f \in \mathcal{B}_{b}(\ob{D})$ and $t>0$, $p_{t}^0f \in C_{b}(\ob{D})$.
\end{theorem}
\begin{proof}
Fix an $f \in \mathcal{B}_{b}(\ob{D})$ and $t>0$. It follows from Lemma~\ref{exit1} that
\begin{align*}
\varlimsup_{n \to \infty}\sup_{x \in K}|p_{t}^{0}f(x)-p_{t}^{n}f(x)|&\le \|f\|_{L^\infty(\ob{D},m)}\times \varlimsup_{n \to \infty} \sup_{x \in K}P_{x}(t \ge \tau_n) =0
\end{align*}
for any compact subset $K$ of $\ob{D}$.
We may assume $K \subset O_1$. From~Lemma~\ref{lem:partstrf2}, $p_{t}^{n}f \in C_{b}(O_n)$ for all $n \in \N$. Hence, we have the claim.
\qed\end{proof}

\begin{proof}[Proof of Theorem~\ref{thm:doubf}~(i)]
Let $f \in \mathcal{B}_{b}(\ob{D})$ and $s,t>0$ with $s<t$.
We write $A_{t}$ for $\int_{0}^{t} \beta(X_{s}^{0})\,dL_s$. Fix a compact subset $K$ of $\ob{D}$. By the Markov property of $X^0$,
\begin{align*}
&\sup_{x \in K}\left| p_{t}f(x)-p^{0}_{s}p_{t-s}f(x)  \right|\\
&=\sup_{x \in K} \left| E_{x}[\exp(-A_t)f(X_{t}^0)]-E_{x}[p_{t-s}f(X_s^{0})]\right| \\
&=\sup_{x \in K}\left|  E_{x} [\exp(-A_s)E_{X_{s}^0}[\exp(-A_{t-s})f(X_{t-s}^0)]]-E_{x}[p_{t-s}f(X_s^{0})] \right|\\
&\le \|f\|_{L^\infty(\ob{D},m)} \times \sup_{x \in K}E_{x}[1-\exp(-A_s )].
\end{align*}
Letting $s \to 0$, we obtain the claim from Proposition~\ref{blt1} and Theorem~\ref{cor:corstf}.
\qed\end{proof}

\subsection{Proof of Theorem~\ref{thm:doubf}~(ii) and Corollary~\ref{thm:3}}
We note that $\beta$ is assumed to be such that $\sg\text{-}\essinf_{\partial D} \beta>0$ in this section.
From Proposition~\ref{prop:1}, we can define the following Dirichlet form on $L^{2}(D,m)$:
\begin{align*}
\cA(f,g)&=\frac{1}{2}\int_{D}(\nabla f,\nabla g)\,dx+\int_{\partial D}\tilde{f} \tilde{g}\,\beta d\sg,\\
\mathcal{D}(\cA)&=H^{1}(D) \cap \left\{ f \in H^{1}(D) \relmiddle| \int_{\partial D}\tilde{f}^2\, \beta d\sg<\infty \right\},
\end{align*}
where $\tilde{f}$ is a $\Capa_{\ob{D}}$-quasi continuous version of $f \in H^{1}(D)$. From \cite[Theorem~6.1.2]{FOT} and Proposition~\ref{prop:1}, $(\cA,\mathcal{D}(\cA))$ is a regular Dirichlet form on $L^{2}(\ob{D},m)$.
Let $\{T_t\}_{t>0}$ be the $L^2$-semigroup associated with $(\cA,\mathcal{D}(\cA) )$. From Lemma~\ref{lem:mazja} below, $\{T_t\}_{t>0}$ has a ultracontractivity i.e. for any $t>0$ and $f \in L^{1}(\ob{D},m)$, we have $T_{t}f  \in L^{\infty}(\ob{D},m)$.
\begin{lemma}\label{lem:mazja}
There exists $a_0=a_{0}(d)>0$ such that
\begin{equation*}
\|f\|_{L^{2d/(d-1)}(D,m)} \le a_{0}\cA_{1}(f,f),\quad f \in \mathcal{D}(\cA).
\end{equation*}
In particular, there exists $a_1>0$ such that for all $t>0$ and $f \in L^{1}(\ob{D},m)$,  we have $T_{t}f \in L^{\infty}(\ob{D},m)$ and
\begin{equation*}
\|T_tf\|_{L^{\infty}(\ob{D},m)} \le a_{1}e^{t}t^{-d}\|f\|_{L^{1}(\ob{D},m)}.
\end{equation*}
\end{lemma}
\begin{proof}
Using Maz'ya's result \cite[Corollary, p.~319]{M}, we have
\begin{equation}
\|f\|_{L^{2d/(d-1)}(D,m)} \le a_{0}\cA_{1}(f,f) \label{eq:mazya}
\end{equation}
for all $f \in H^{1}(D) \cap C_{c}(\ob{D})$. From Proposition~\ref{prop:1}, $H^{1}(D) \cap C_{c}(\ob{D})=\mathcal{D}(\cA) \cap C_{c}(\ob{D})$. Since $(\cA,\mathcal{D}(\cA))$ is regular on $L^{2}(\ob{D},m)$, \eqref{eq:mazya} holds for all $f \in \mathcal{D}(\cA)$. 
\qed\end{proof}

Since $\{T_t\}_{t>0}$ is ultracontractive, each $T_t$ admits an integral kernel $p_{t}(x,y)$. One can show that $p_{t}(x,y)$ has Gaussian estimates following the lines of the proof of \cite[Theorem~4.4]{AT}. See also \cite[Theorem~5.3]{AW}.
\begin{theorem}\label{thm:hke}
There exists positive constants $a_2,a_3>0$ such that 
\begin{equation*}
p_{t}(x,y) \le a_{2}e^{t}t^{-d}\exp(-|x-y|^2/a_{3}t)
\end{equation*}
for all $t>0$ and $m$-a.e. $(x,y)\in \ob{D} \times \ob{D}$.
\end{theorem}

From \cite[Theorem~6.1.1]{FOT} and Proposition~\ref{prop:1}, the Dirichlet form of $Y=( \{Y_t\}_{t \ge 0}, \{P_x\}_{x \in \ob{D}})$ is $( \cA ,\mathcal{D} (\cA))$. Hence, for any $t>0$ and $f \in \mathcal{B}_{b}(\ob{D}) \cap L^{2}(\ob{D},m)$, we have $p_{t}f =T_{t}f$ $m$-a.e.
\begin{lemma}\label{lem:hke3}
For any $t>0$ and $r>4$,
\begin{equation*}
\sup_{x \in \ob{D}}p_{t}\bone_{\ob{D} \setminus B(x,r)}(x) \le a_{4}r^{-a_5},
\end{equation*}
where $a_{4}$ and $a_5$ are positive constants independ of $r$.
\end{lemma}
\begin{proof}
Fix $x \in \ob{D}$. By Theorem~\ref{thm:hke} and straightforward calculation,
\begin{equation}
p_{t}\bone_{ \ob{D} \setminus B(x,r)}(y)  \le a_{4}r^{-a_5} \label{eq:eqtail}
\end{equation}
for $m$-a.e. $y \in B(x,r/4) \cap \ob{D}$. 
Here, $a_{4}$ and $a_5$ are positive constants independ of $r,x,y$. From Theorem~\ref{thm:doubf}~(i)],  we have $p_{t}\bone_{\ob{D} \setminus B(x,r)} \in C_{b}(\ob{D})$. Hence
\begin{equation*}
p_{t}\bone_{ \ob{D} \setminus B(x,r)}(x)  \le a_{4}r^{-a_5}.
\end{equation*}
Taking supremum in $x$, we complete the proof.
\qed\end{proof}

\begin{proof}[Proof of Theorem~\ref{thm:doubf}~(ii)]
It suffices to show that $p_{t}( C_{c}(\ob{D})) \subset C_{\infty}(\ob{D})$. For any $f \in C_{c}(\ob{D})$, we have $p_{t}f \in C_{b}(\ob{D})$ from Theorem~\ref{thm:doubf}~(i). Therefore, it remains to show $p_{t}f$ vanishes at infinity. It follows from Lemma~\ref{lem:hke3} that 
\begin{align*}
|p_{t}f(x)| 
&\le \sup_{y \in \ob{D} \cap B(x,r)} |f(y)|+\|f\|_{L^\infty(\ob{D},m)} \times p_{t} \bone_{\ob{D}  \setminus B(x,r)}(x)\\
&\le  \sup_{y \in \ob{D} \cap B(x,r)} |f(y)|+\|f\|_{L^\infty(\ob{D},m)} \times a_{4}r^{-a_5}
\end{align*}
for each $x \in \ob{D}$.
By letting $|x| \to \infty$ and  then $r \to \infty$, we obtain the claim.
\qed\end{proof}

\begin{proof}[Proof of Corollary~\ref{thm:3}]
Note that 
$$\lim_{x \in \ob{D},\,|x| \to \infty}m(D \cap B(x,r)) \to 0$$
holds for any $r>0$. It follows from Lemma~\ref{lem:mazja} and Lemma~\ref{lem:hke3} that
\begin{align*}
p_{t}\bone_{\ob{D}}(x)
&= p_{t}\bone_{\ob{D} \cap B(x,r)}(x)+ p_{t} \bone_{\ob{D}  \setminus B(x,r)}(x)\\
&\le a_{1}e^{t} \times  m(D \cap B(x,r)) + a_{4}r^{-a_5}
\end{align*}
for any $t>0$ and $x \in \ob{D}$. By letting $|x| \to \infty$ and  then $r \to \infty$, we have $p_{t}\bone_{\ob{D}} \in C_{\infty}(\ob{D})$ for any $t>0$. It is easy to see $R_{\alpha}\bone_{\ob{D}} \in C_{\infty}(\ob{D})$ for any $\alpha>0$. It is clear that $Y$ is irreducible. Using these properties and Theorem~\ref{thm:doubf}, the proof of Corollary~\ref{thm:3} is complete from \cite[Theorem~1.1]{T}.
\qed\end{proof}

\section{Proof of auxiliary lemmas}

\begin{definition}[Bounded Lipschitz domain]\label{def:bld}
Let $D \subset \R^d$ be a bounded connected open subset. $D$ is called a {\it bounded Lipschitz domain} if there exist positive constants $\dl^{\ast}$, $M^{\ast}$ such that for each $x_0 \in \partial \Omega$ there exist a neighborhood $U_{x_0}$ of $x_0$, local coordinates $y=(y',y_d) \in \mathbb{R}^{d-1} \times \mathbb{R}$, with $y=0$ at $x_0$, and a  Lipschitz continuous function $f_{x_0}:\mathbb{R}^{d-1} \to \mathbb{R}$, such that
\begin{equation*}
D \cap U_{x_0}=\{(y',y_N) \in D \cap U_{x_0} \mid y' \in B(0,\dl^{\ast}),\, y_N>f(y') \},\quad \text{Lip}(f) \le M^{\ast},
\end{equation*}
where we define $\text{Lip}(f)=\inf \{L \ge 0 \mid |f(x)-f(y)| \le L|x-y|,\, x,y \in B(0,\dl^{\ast}) \}$. 
\end{definition}

\begin{lemma}\label{lem:blda}
Every bounded Lipschitz domain $D$ satisfies the following:
\begin{enumerate}[({A'}.1)]
\item There exist $\dl$, $M>0$ such that for any $a \in \partial D$, there are its neighbourhood $W_a$ in $\R^d$, and one to one mapping  $\Psi_a$ from $B(\dl)$ onto $W_a$ such that $\Psi_{a}(0)=a$, $\Psi_a(B_{+}(\dl))=W_a \cap D$, $\max\{ \text{Lip}(\Psi_a), \text{Lip}(\Psi_{a}^{-1})\} \le M$. 
\end{enumerate}
In particular, $D$ satisfies (A.1), (A.2), and (A.3).
\end{lemma}
\begin{proof}
We take positive constants $\dl^{\ast},M^{\ast}$ as in Definition~\ref{def:bld}.  Then, for each $x_0 \in \partial D$, there exists a Lipschitz continuous function $f_{x_0}:\mathbb{R}^{d-1} \to \mathbb{R}$ with $\text{Lip}(f_{x_0}) \le M^{\ast}$.  Define $\Psi_{x_0}: B(\dl^{\ast}) \to \R^d$ and $W_{x_0}$ by
\begin{align*}
\Psi_{x_0}(x',x_d)=(x',x_d+f(x')),\quad W_{x_0}=\Psi_{x_0}(B(\dl^{\ast})).
\end{align*}
Then, $\dl=\dl^{\ast}$, $M=M^{\ast}+1$, $W_{x_0}$, and $\Psi_{x_0}$ satisfy the required condition in (A'.1). Clearly, (A'.1) implies (A.1) and (A.2). Since bounded Lipschitz domains satisfy the segment condition, we obtain (A.3) from \cite[Theorem~3.22]{Ad}.
\qed\end{proof}

\begin{proposition}
Condition~\ref{as:2} implies Condition~\ref{as:1}.
\end{proposition}
\begin{proof}
Take increasing bounded open subsets $\{A_n\}_{n=1}^{\infty}$ of $\partial D$ such that $\partial D=\bigcup_{n=1}^{\infty}A_{n}$. Since $A_1$ is bounded, there exists a bounded open subset $U_1$ of $\R^d$ such that $A_{1} \subset U_{1}$ and $U_{1} \cap D$ is a bounded Lipschitz domain of $\R^d$. From Lemma~\ref{lem:blda}, for any $a \in A_1$, there exist bi-Lipschitz mapping $\Psi_{a}$ required in (A.1) and $M_1>0$ such that 
\begin{align*}
\sup_{a \in A_1} \max \{\text{Lip}(\Psi_{a}), \text{Lip}(\Psi_{a}^{-1})\}\le M_{1}.
\end{align*}
Since $A_{2} \setminus A_1$ is bounded, from Lemma~\ref{lem:blda}, for any $a \in A_2 \setminus A_1$, there exist bi-Lipschitz mapping $\Psi_{a}$ required in (A.1) and $M_2>0$ such that 
$
\sup_{a \in A_2 \setminus A_{1}} \max \{\text{Lip}(\Psi_{a}), \text{Lip}(\Psi_{a}^{-1})\}\le M_{2}.
$
In the same manner, for any $n \ge 1$ and $a \in A_{n+1} \setminus A_{n}$, we can find bi-Lipschitz mapping $\Psi_{a}$ required in (A.1) and $M_n>0$ such that 
\begin{align*}
\sup_{a \in A_n \setminus A_{n-1}} \max \{\text{Lip}(\Psi_{a}), \text{Lip}(\Psi_{a}^{-1})\}\le M_{n}.
\end{align*}
These bi-Lipschitz mappings $\{\Psi_{a}\}_{a \in \partial D}$ clearly satisfy the conditions (A.1) and (A.2). 
Since any $a \in \partial D$ belongs to the boundary of some bounded Lipschitz domain, we can check (A.3) from \cite[Theorem~3.22]{Ad}.
\qed\end{proof}
Although the following lemma is a slightly modification of \cite[Theorem~2.1]{EPH}, we give a proof for reader's convenience.
\begin{lemma}\label{lem:EPH}
Let $D \subset \R^d$ be a bounded Lipschitz domain. For $\eps \in (0,1)$, let $D_{\eps}=\{x \in D \mid \text{dist}(x,\partial D)<\eps\}$. Then, there exist $\eps_0>0$ and $c=c(d,D)>0$ such that 
\begin{equation*}
\frac{1}{\eps}\int_{D_\eps}s^{-d/2}\exp(-|x-y|^2/s)\,dy \le c/\sqrt{s}.
\end{equation*}
for any $x \in \ob{D}, \eps \in (0,\eps_0)$, and $s \in (0,1]$.
\end{lemma}
\begin{proof}
Integration by parts gives
\begin{align} \label{eq:cy}
&\int_{D_\eps}s^{-d/2}\exp(-|x-y|^2/s)\,dy
=\int_{0}^{\infty}s^{-d/2}\exp(-r^2/s)\,d\{m(B(x,r) \cap D_\eps)\} \\
&=2s^{-(d+2)/2}\int_{0}^{\infty}m(B(x,r) \cap D_\eps)\exp(-r^2/s)r\,dr \notag.
\end{align}
As $D$ is a bounded Lipschitz domain, there exist positive constants $c=c(D)$ and $\eps_0$ such that
\begin{equation}\label{eq:cy2}
m(B(x,r) \cap D_\eps)\le c \eps r^{d-1} 
\end{equation}
for any $x \in \ob{D}$ and $r>0$. This can be shown rigorously by working in a local corrdinate syetem. In fact, it is easy to verify that the set $B(x,r) \cap D_\eps$ is contained in a cylinder with base area of order $r^{d-1}$ and height of $\eps>0$. Combining \eqref{eq:cy} with \eqref{eq:cy2}, we obtain 
\begin{align*} 
&\sup_{x \in \ob{D}}\sup_{ \eps \in (0,\eps_0)}\frac{1}{\eps}\int_{D_\eps}s^{-d/2}\exp(-|x-y|^2/s)\,dy
\\
&\le 2c s^{-(d+2)/2}\int_{0}^{\infty}\exp(-r^2/s)r^d\,dr \\
&=2c s^{-(d+2)/2} \times \left\{\frac{(d-1)s}{2} \int_{0}^{\infty}\exp(-r^2/s)r^{d-2}\,dr\right\} \\
&=2c s^{-(d+2)/2} \times \frac{(d-1)s}{2} \times \frac{(d-3)s}{2} \int_{0}^{\infty}\exp(-r^2/s)r^{d-4}\,dr.
\end{align*}
By iterating this procedure, we obtain the next inequalities. If $d$ is even, it holds that
\begin{align} 
&\sup_{x \in \ob{D}}\sup_{ \eps \in (0,\eps_0)}\frac{1}{\eps}\int_{D_\eps}s^{-d/2}\exp(-|x-y|^2/s)\,dy
 \label{eq:eqlast1} \\
&\le 
2c s^{-1} \times \frac{(d-1)\times  (d-3) \times \cdots \times 1}{2^{d/2}} \times \int_{0}^{\infty}\exp(-r^2/s)\,dr \notag \\
&= 2c s^{-1/2} \times \frac{(d-1)\times  (d-3) \times \cdots \times 1}{2^{d/2}} \times \int_{0}^{\infty} \exp(-r^2)\,dr. \notag
\end{align}
If $d$ is odd, it holds that
\begin{align} 
&\sup_{x \in \ob{D}}\sup_{ \eps \in (0,\eps_0)}\frac{1}{\eps}\int_{D_\eps}s^{-d/2}\exp(-|x-y|^2/s)\,dy \label{eq:eqlast2}
\\
&\le 
2c s^{-3/2} \times \frac{(d-1)\times  (d-3) \times \cdots \times 1}{2^{(d-1)/2}} \times \int_{0}^{\infty}\exp(-r^2/s)r \,dr \notag \\
&=2c s^{-1/2} \times \frac{(d-1)\times  (d-3) \times \cdots \times 1}{2^{(d-1)/2}} \times \int_{0}^{\infty}\exp(-r^2)r \,dr. \notag
\end{align}
\eqref{eq:eqlast1} and \eqref{eq:eqlast2} complete the proof.
\qed
\end{proof}

\begin{acknowledgement}
The author would like to thank Professor Masayoshi Takeda for detailed discussions and helpful support. He would like to thank referees for their valuable comments and suggestions which improve the quality of the paper. He would also like to thank Dr. Masaki Wada for encouragement. 
\end{acknowledgement}



\begin{bibdiv}
\begin{biblist}

\bib{Ad}{book}{
   author={Adams, Robert A.},
   author={Fournier, John J. F.},
   title={Sobolev spaces},
   series={Pure and Applied Mathematics (Amsterdam)},
   volume={140},
   edition={2},
   publisher={Elsevier/Academic Press, Amsterdam},
   date={2003},
   pages={xiv+305},
   isbn={0-12-044143-8},
}


\bib{AT}{article}{
   author={Arendt, W.},
   author={ter Elst, A. F. M.},
   title={Gaussian estimates for second order elliptic operators with
   boundary conditions},
   journal={J. Operator Theory},
   volume={38},
   date={1997},
   number={1},
   pages={87--130},
   issn={0379-4024},
}
\bib{AW}{article}{
   author={Arendt, Wolfgang},
   author={Warma, Mahamadi},
   title={The Laplacian with Robin boundary conditions on arbitrary domains},
   journal={Potential Anal.},
   volume={19},
   date={2003},
   number={4},
   pages={341--363},
   issn={0926-2601},
}



\bib{BH}{article}{
   author={Bass, Richard F.},
   author={Hsu, Pei},
   title={Some potential theory for reflecting Brownian motion in H\"older and
   Lipschitz domains},
   journal={Ann. Probab.},
   volume={19},
   date={1991},
   number={2},
   pages={486--508},
   issn={0091-1798},
}

\bib{B}{article}{
   author={Biegert, Markus},
   title={On traces of Sobolev functions on the boundary of extension
   domains},
   journal={Proc. Amer. Math. Soc.},
   volume={137},
   date={2009},
   number={12},
   pages={4169--4176},
   issn={0002-9939},
}

\bib{BKK}{article}{
   author={Bogdan, Krzysztof},
   author={Kumagai, Takashi},
   author={Kwa\'snicki, Mateusz},
   title={Boundary Harnack inequality for Markov processes with jumps},
   journal={Trans. Amer. Math. Soc.},
   volume={367},
   date={2015},
   number={1},
   pages={477--517},
   issn={0002-9947}
}

\bib{CF}{article}{
   author={Chen, Zhen-Qing},
   author={Fan, Wai-Tong},
   title={Systems of interacting diffusions with partial annihilation
   through membranes},
   journal={Ann. Probab.},
   volume={45},
   date={2017},
   number={1},
   pages={100--146},
   issn={0091-1798},
}

\bib{CK}{article}{
   author={Chen, Zhen-Qing},
   author={Kuwae, Kazuhiro},
   title={On doubly Feller property},
   journal={Osaka J. Math.},
   volume={46},
   date={2009},
   number={4},
   pages={909--930},
   issn={0030-6126},
}



\bib{C}{article}{
   author={Chung, K. L.},
   title={Doubly-Feller process with multiplicative functional},
   conference={
      title={Seminar on stochastic processes, 1985},
      address={Gainesville, Fla.},
      date={1985},
   },
   book={
      series={Progr. Probab. Statist.},
      volume={12},
      publisher={Birkh\"auser Boston, Boston, MA},
   },
   date={1986},
   pages={63--78},
}

	
\bib{EE}{book}{
   author={Edmunds, D. E.},
   author={Evans, W. D.},
   title={Spectral theory and differential operators},
   series={Oxford Mathematical Monographs},
   note={Oxford Science Publications},
   publisher={The Clarendon Press, Oxford University Press, New York},
   date={1987},
   pages={xviii+574},
   isbn={0-19-853542-2},
}



\bib{EG}{book}{
   author={Evans, Lawrence C.},
   author={Gariepy, Ronald F.},
   title={Measure theory and fine properties of functions},
   series={Textbooks in Mathematics},
   edition={Revised edition},
   publisher={CRC Press, Boca Raton, FL},
   date={2015},
   pages={xiv+299},
   isbn={978-1-4822-4238-6},
}


\bib{FOT}{book}{
   author={Fukushima, Masatoshi},
   author={Oshima, Yoichi},
   author={Takeda, Masayoshi},
   title={Dirichlet forms and symmetric Markov processes},
   series={De Gruyter Studies in Mathematics},
   volume={19},
   edition={Second revised and extended edition},
   publisher={Walter de Gruyter \& Co., Berlin},
   date={2011},
   pages={x+489},
   isbn={978-3-11-021808-4}
}

\bib{FT0}{article}{
   author={Fukushima, Masatoshi},
   author={Tomisaki, Matsuyo},
   title={Reflecting diffusions on Lipschitz domains with cusps---analytic
   construction and Skorohod representation},
   journal={Potential Anal.},
   volume={4},
   date={1995},
   number={4},
   pages={377--408},
   issn={0926-2601},
}
	
\bib{FT}{article}{
   author={Fukushima, Masatoshi},
   author={Tomisaki, Matsuyo},
   title={Construction and decomposition of reflecting diffusions on
   Lipschitz domains with H\"older cusps},
   journal={Probab. Theory Related Fields},
   volume={106},
   date={1996},
   number={4},
   pages={521--557},
   issn={0178-8051},
}


\bib{HKT}{article}{
   author={Haj\l asz, Piotr},
   author={Koskela, Pekka},
   author={Tuominen, Heli},
   title={Sobolev embeddings, extensions and measure density condition},
   journal={J. Funct. Anal.},
   volume={254},
   date={2008},
   number={5},
   pages={1217--1234},
   issn={0022-1236},
}

\bib{EPH}{book}{
   author={Hsu, P}
   title={Reflecting Brownian motion, boundary local time and the neumann problem, Ph. D. thesis, Stanford University, 1994}
}


  \bib{M}{book}{
   author={Maz'ya, Vladimir},
   title={Sobolev spaces with applications to elliptic partial differential
   equations},
   series={Grundlehren der Mathematischen Wissenschaften [Fundamental
   Principles of Mathematical Sciences]},
   volume={342},
   edition={Second, revised and augmented edition},
   publisher={Springer, Heidelberg},
   date={2011},
   pages={xxviii+866},
   isbn={978-3-642-15563-5},
}
\bib{Mo}{article}{
   author={Moser, J\"urgen},
   title={A new proof of De Giorgi's theorem concerning the regularity
   problem for elliptic differential equations},
   journal={Comm. Pure Appl. Math.},
   volume={13},
   date={1960},
   pages={457--468},
   issn={0010-3640},
}
\bib{ST}{article}{
   author={Stampacchia, Guido},
   title={Equations elliptiques du second ordre \`{a} coeeficients discontinuous},
   book={
      publisher={S\'{e}minar sur les equations aux de\'{r}iv\'{e}es partielles, Coll\`{e}ge de France},
   },
   date={1963},
}

\bib{SV}{book}{
   author={Stroock, Daniel W.},
   author={Varadhan, S. R. Srinivasa},
   title={Multidimensional diffusion processes},
   series={Classics in Mathematics},
   note={Reprint of the 1997 edition},
   publisher={Springer-Verlag, Berlin},
   date={2006},
   pages={xii+338},
   isbn={978-3-540-28998-2}
}


\bib{SG}{article}{
   author={Shin, Jiyong},
   author={Trutnau, Gerald},
   title={On the stochastic regularity of distorted Brownian motions},
   journal={Trans. Amer. Math. Soc.},
   volume={369},
   date={2017},
   number={11},
   pages={7883--7915},
   issn={0002-9947},
}
	
\bib{T}{article}{
author={Takeda, Masayoshi},
title={Compactness of symmetric Markov semi-groups and boundedness of eigenfuntions, to appear in Trans. Amer. Math. Soc.}
}	
\end{biblist}
\end{bibdiv}
\end{document}